\documentclass[11pt]{amsart} 

\usepackage{amssymb,amsmath,url}
\usepackage{hyperref,color}
\usepackage{amsfonts}
\usepackage{amscd}
\usepackage{graphicx}
\usepackage[all]{xy}
\usepackage{enumerate}

\DeclareMathOperator*{\colim}{colim}
\theoremstyle{plain}
  \newtheorem*{thm*}{Theorem}
 \theoremstyle{plain}
\newtheorem{thm}{Theorem}[section]
  \theoremstyle{plain}
  \newtheorem{prop}[thm]{Proposition}
  \theoremstyle{plain}
  \newtheorem{cor}[thm]{Corollary}
  \theoremstyle{plain}
  \newtheorem{lem}[thm]{Lemma}
\theoremstyle{plain}
  \newtheorem{rem}[thm]{Remark}
\theoremstyle{plain}
  \newtheorem{conj}[thm]{Conjecture}
\theoremstyle{plain}
  
\theoremstyle{plain}
  
  \theoremstyle{plain}
  \newtheorem{que}[thm]{Question}
    \theoremstyle{plain}
  \newtheorem{exam}[thm]{Example}

\newcommand{\sm}{\setminus}
\newcommand{\st}{\mathrm{star}}
\newcommand{\bbR}{\mathbb{R}}
\def\co{\colon\thinspace}
\newcommand{\injects}{\hookrightarrow}
\newcommand{\homeo}{\cong}

\newcommand{\isom}{\cong}
\newcommand{\leqs}{\leqslant}
\newcommand{\geqs}{\geqslant}
\newcommand{\cross}{\times}
\newcommand{\heq}{\simeq}
\newcommand{\maps}{\longrightarrow}
\newcommand{\inter}[1]{\stackrel{\circ}{#1}}
\newcommand{\Inter}[1]{\textrm{Inter}(#1)}
\newcommand{\srm}[1]{\stackrel{#1}{\maps}}
\newcommand{\srt}[1]{\stackrel{#1}{\to}}
\newcommand{\notsubset}{\nsubseteq}
\newcommand{\e}{\emph}
\newcommand{\xmaps}{\xrightarrow}

\newcommand{\wt}[1]{\widetilde{#1}}
\newcommand{\nsubset}{\notsubset}

\newcommand{\C}{\mathbb{C}}
\newcommand{\R}{\mathbb{R}}
\newcommand{\tr}[1]{\mathrm{tr}(#1)}

\newcommand{\aq}{/\!\!/}
\newcommand{\X}{\mathfrak{X}}

\renewcommand{\hom}{\mathrm{Hom}}

\newcommand{\F}{\mathtt{F}}

\newcommand{\XC}[1]{\mathfrak{X}_{#1}}

\newcommand{\Z}{\mathbb{Z}}

\newcommand{\SLm}[1]{\mathsf{SL}_{#1}}
\newcommand{\GLm}[1]{\mathsf{GL}_{#1}}

\newcommand{\SUm}[1]{\mathsf{SU}_{#1}}
\newcommand{\PSLm}[1]{\mathsf{PSL}_{#1}}
\newcommand{\Um}[1]{\mathsf{U}_{#1}}

\newcommand{\cp}{\C\p}
\newcommand{\p}{\mathsf{P}} 
\newcommand{\codim}{\mathrm{codim}}

\AtBeginDocument{%
   \def\MR#1{}
}

\makeatother

\begin{document}

\bibliographystyle{amsalpha}

\title[Homotopy of Character Varieties]{Homotopy Groups of Free Group Character Varieties}

\author[C. Florentino]{Carlos Florentino}

\address{Departamento de Matem\'{a}tica, Faculdade de Ci\^{e}ncias da Universidade de Lisboa, Edf. C6, Campo Grande 1749-016 Lisbon, Portugal}
\email{caflorentino@ciencias.ulisboa.pt}

\author[S. Lawton]{Sean Lawton}

\address{Department of Mathematical Sciences, George Mason University,
4400 University Drive,
Fairfax, Virginia  22030, USA}

\email{slawton3@gmu.edu}

\author[D. Ramras]{Daniel Ramras}

\address{Department of Mathematical Sciences, Indiana University-Purdue University Indianapolis, 402 N. Blackford, 
Indianapolis, IN 46202, USA}

\email{dramras@math.iupui.edu}

\thanks{Florentino was supported by the projects  PTDC/MAT-GEO/0675/2012 and EXCL/MAT-GEO/0222/2012, FCT, Portugal.  Lawton and Ramras were supported by Collaboration Grants from the Simons Foundation, USA.  Additionally, Lawton was supported by U.S. National Science Foundation grant DMS-1309376.}



\begin{abstract} Let $G$ be a connected, complex reductive Lie group with maximal compact subgroup $K$, 
and let  $\XC{r}$ denote the moduli space of $G$-- or $K$--valued representations of a rank $r$ free group.  In this article, we develop methods for studying the low-dimensional homotopy groups of these spaces and of their subspaces $\XC{r}^{irr}$ of irreducible representations.

Our main result is that when $G = \GLm{n} (\C)$ or $\SLm{n} (\C)$, the second homotopy group of $\XC{r}$ is trivial.  The proof  depends on a new general position-type result in a singular setting.  This result is proven in the Appendix and may be of independent interest.   

We also obtain new information regarding the homotopy groups of the subspaces $\XC{r}^{irr}$.
Recent work of Biswas and Lawton determined $\pi_1 (\XC{r})$ for general $G$, and we describe $\pi_1 (\XC{r}^{irr})$.  Specializing to the case $G = \GLm{n} (\C)$, we explicitly compute the homotopy groups of the smooth locus  $\XC{r}^{sm} = \XC{r}^{irr}$ in a large range of dimensions, finding that they exhibit Bott Periodicity. 

As a further application of our methods (and in particular our general position result)
we obtain new results regarding centralizers of subgroups of $G$ and $K$, motivated by a question of Sikora.  
 
 Additionally, we use work of Richardson to solve a conjecture of Florentino--Lawton about the singular locus of $\XC{r}$, and we give a topological proof that for $G= \GLm{n} (\C)$ or $\SLm{n} (\C)$, the space $\XC{r}$ is not a rational Poincar\'e Duality Space for $r\geqs 4$ and $n=2$.
\end{abstract}

\maketitle

\tableofcontents

\section{Introduction}
Given an irreducible singular algebraic variety $X$ with singular set $X^{sing}$, what is the relationship between the topology of $X$ and that of its smooth locus $X^{sm} = X \setminus X^{sing}$?  In the case of a smooth manifold $M$ and a smooth submanifold $N\subset M$ of codimension $c$, transversality shows that the inclusion $M\setminus N \injects M$ is $(c-1)$--connected (that is, an isomorphism on homotopy groups $\pi_i$ for $i=0,...,c-2$, and surjective on $\pi_{c-1}$).  However, the corresponding statement for the inclusion $X^{sm}  \injects X$ fails in general.

In this article, we give general local conditions guaranteeing that the inclusion $X^{sm} \injects X$ is $2$--connected, and we prove that these conditions are satisfied for certain character varieties. We go on to show that in some cases, this inclusion is \e{not} $3$--connected despite the singular locus lying in codimension at least 4, illustrating the failure mentioned above.  This involves a calculation of the second homotopy groups of these character varieties and of their smooth loci, and also leads to new results on their homology. 

Our main result, which is proven using these ideas, is as follows:
\begin{thm*}[Theorem~\ref{pi2}] 
Let $G_n = \GLm{n}(\C), \SLm{n}(\C), \SUm{n}$
or $\Um{n}$. Then  $\pi_{2}(\XC{r}(G_n))= 0$, where $\XC{r} (G_n)$ is the character variety 
associated to the free group $\F_r$ of rank $r$ and the group $G_n$.
\end{thm*} 
While character varieties of free groups and their topology have been studied extensively, previous computational results have focused on fundamental groups \cite{BiLa, BiLaRa, Lawton-Ramras} or on rational homology in the special cases $G = \SUm{2}$~\cite{Ba0}.  Theorem~\ref{pi2} appears to be the first systematic calculation of higher homotopy (excepting Steinberg's results for the case of $G\aq G$; see Remark~\ref{Ste-rmk}).

We now describe the  varieties studied in this paper.  Consider a rank $r$ free group $\F_r$, a Lie group $G$, and the set of group homomorphisms $\hom(\F_r,G)$. Being naturally in bijection with the Cartesian product $G^r$, $\hom(\F_r,G)$ has a natural smooth manifold structure. The group $G$ acts analytically on $\hom(\F_r,G)$ by conjugation, but the orbit space, $\hom(\F_r,G)/G$, is not generally Hausdorff when $G$ is not compact. 
We consider instead the subspace of closed orbits, also called the {\it polystable} locus and denoted by 
$\hom(\F_r,G)^*$ (it is generally not a subvariety, only a constructible set), and we study the corresponding quotient 
$\XC{r}(G):=\hom(\F_r,G)^*/G$, often referred to as the $G$--{\it character variety} of $\F_r$.

When $G$ is compact, $\XC{r}(G)$ coincides with the orbit space $\hom(\F_r,G)/G$, but when $G$ is a connected, complex reductive affine algebraic group (for short, a connected reductive $\C$--group), $\XC{r}(G)$ is homeomorphic to the (affine) Geometric Invariant Theory (GIT) quotient $\hom(\F_r,G)\aq G$ equipped with the Euclidean topology (for a detailed proof, see~\cite[Theorem 2.1]{FlLa3})\footnote{In fact, $\X_r(G)$ is the categorical quotient in the category of affine varieties, Hausdorff spaces or complex analytic varieties \cite{Lu2,Lu3}.}. 
Despite its name, it is not always an algebraic set nor does it always parametrize traditional characters, although whenever $G$ is a connected reductive $\C$--group, GIT implies $\XC{r}(G)$ is an irreducible algebraic set (a variety), and for some classical groups, it does parametrize characters (see \cite{Si5}, or Appendix A in \cite{FlLa2b}).  We note that there is a deformation retraction from the non-Hausdorff
space $\hom(\F_r, G)/G$ to $\XC{r} (G)$ (see Proposition \ref{def-retr}), so that our results on homotopy and homology apply to $\hom(\F_r, G)/G$ as well.

Now let $G$ be a connected reductive $\C$--group.   Recall that an affine algebraic group is reductive if and only if its radical is a torus.  It is well known that an affine algebraic $\C$--group is reductive if and only if it is the complexification of a maximal compact subgroup $K < G$.  The derived subgroup $DG = [G, G]$ is the maximal semisimple subgroup of $G$, and is a connected reductive $\C$--group with maximal compact subgroup $DK = [K, K]$.
 
Since $\XC{r}(G)$ is an irreducible algebraic set, it is path connected, and its zeroth homotopy group, $\pi_0(\XC{r}(G))$, is trivial.  In \cite{BiLa}, the fundamental group 
$\pi_1(\XC{r}(G))$ is shown to be isomorphic to $\pi_1(G/DG)^r$.  By \cite{FlLa}, these results remain true when $G$ is instead a compact connected Lie group.

In this paper we begin the study of higher homotopy groups of these moduli spaces and of 
their irreducible loci $\XC{r}(G)^{irr}$.
The irreducible locus $\XC{r}(G)^{irr}$ (a representation $\rho:\F_r \to G$ is irreducible if its image $\rho(\F_r)$ is not contained in a proper parabolic subgroup of $G$) also has a natural GIT interpretation.
Indeed, $\rho$ is irreducible if and only
if it is a GIT stable point of the $G$-action on the representation
space (see \cite{CaFl}). Therefore, $\XC{r}(G)^{irr}$ is actually the moduli space of stable
representations\footnote{Here we use the slight generalization of the notion of stability,
in affine GIT, introduced by Richardson \cite{Ri}.}  of $\F_r$ into $G$.  Such moduli spaces have been studied extensively for \e{closed surface groups} (analogous representation spaces where $\F_r$ is replaced by $\pi_1(X)$ or a certain central extension, for a compact Riemann surface $X$). 
Higher homotopy groups of these spaces were first studied in~\cite{BGG, DU}, using gauge theoretic methods. For $G = \GLm{n}(\C), \SLm{n}(\C), \SUm{n}$, or $\Um{n}$ we establish a periodicity result 
for  $\pi_* \XC{r}(G)^{irr}$ 
(Theorem~\ref{periodicity}) analogous to~\cite[Theorem 4.2(3)]{BGG} and~\cite[Theorem 3.1]{DU}.  All three of these results establish isomorphisms between the homotopy groups of a moduli space of stable bundles and the homotopy of an associated gauge group.
In the case $G=\Um{n}$, Lawson \cite{Law} showed that  the \e{stabilized} character variety $\colim_{n\to \infty} \XC{r} (\Um{n})$ is homotopy equivalent to $(S^1)^r$, but his methods do not give information about the spaces $\XC{r}(\Um{n})$ themselves.  

We begin in Section \ref{charvar-sec} with general relationships between $\XC{r}(G)$, $\XC{r}(DG)$, $\XC{r}(K)$, $\XC{r}(DK)$ and their various loci:  irreducible, reducible, good, smooth, and singular.  
In Section~\ref{fibration}, we use a result from \cite{BiLaRa} to show that $$\pi_k(\XC{r}(G))\cong\pi_k(\XC{r}(DG))\cong\pi_k(\XC{r}(K))\cong\pi_k(\XC{r}(DK)),$$ for $k\geqs 2$.  Next, Section \ref{sah} shows that 
$$\pi_1 (\XC{r} (G)^{irr}) \isom \pi_1 (G/DG)^r \times \pi_1 (\XC{r} (DG)^{irr}),$$ 
and that $\pi_1 (\XC{r} (DG)^{irr})$ is a quotient of the finite Abelian group $\pi_1 (DG)^r$.  

In Section \ref{linear-sec}, we specialize to the case where $G = \GLm{n}, \SLm{n}, \Um{n}$, or $\SUm{n}$.  In these cases, results in \cite{FlLa2} identify the smooth locus $\XC{r}(G)^{sm}$ with the locus of irreducible characters. 
After a close analysis of the reducible locus, we prove that $\pi_2 (\XC{r} (G)^{sm})$ is isomorphic to $\mathbb{Z}/n\mathbb{Z}$.  We use Bott Periodicity for $G$ to calculate $\pi_k (\XC{r} (G)^{sm})$ in a range of dimensions.
Going further, we show that the natural map $\XC{r}(G)^{sm}\hookrightarrow \XC{r}(G)$ is $2$--connected, and we deduce from this that  $\pi_{2}(\XC{r}(G)) = 0$ for all $r$.   

Our connectivity result is based on the intuition that maps from 2--spheres can be homotoped off of the singular locus, due to its high codimension.  However, the ambient space $\XC{r}(G)$ is singular, so we need a singular analogue of this general position argument.  Such a result is stated in Proposition \ref{trans-prop} and proved in the Appendix, and may be of independent interest.   While the hypotheses of this result do not explicitly refer to codimension, we use codimension arguments, together with the Conner Conjecture (as proven by Oliver \cite{Oliver}) to verify the hypotheses for the case of $\XC{r}(G)$.  

As an application of our topological methods to the the study of Lie groups (Section \ref{app-sec}), we show that if $\pi_1(DG)\not=1$, there exists a finitely generated subgroup $H\leqs G$ with disconnected $PG$--centralizer.  Sikora asked whether this was always true for connected reductive
$\C$--groups (excepting $\SLm{n}$ and $\GLm{n}$) with the additional assumption that $H$ be irreducible [Sik12, Question 19].  Our topological approach provides a new viewpoint on the problem.
 
As our study of homotopy demands attention to the singular locus, in Section \ref{sing-sec} we resolve a conjecture of Florentino-Lawton (see \cite{FlLa2}) that states, in part, that the reducible locus in $\XC{r}(G)$ is always singular for any reductive $G$ (not just for $\SLm{n}$ or $\GLm{n}$).
This relies on work of Richardson (see \cite{Ri}) concerning the case of semisimple $G$.  This further shows that whenever $G$ does not have property CI (see Section \ref{sing-sec} for the definition) then the irreducible locus $\XC{r} (G)^{irr}$ has a non-empty orbifold singular locus (the only known CI groups are $\SLm{n}$ and $\GLm{n}$, and many others are known  not to be CI). 
 
In Section \ref{PP}, we analyze the structure of the Poincar\'e polynomial of $\XC{r} (\SUm{2})$ (as computed in \cite{Ba0}), and give a new proof that these spaces are generally  not manifolds with boundary.  This was first established using algebro-geometric arguments in~\cite{FlLa2}\footnote{Section 8 was originally included as Appendix B in \cite{FlLa2b}.}.  In fact, we prove the stronger result that $\XC{r} (\SUm{2})$ is  generally not a rational Poincar\'e Duality Space.

\section*{Acknowledgments}
Some of this paper was completed at the Centre de Recerca Matem\`atica, Barcelona (Summer 2012) and some of it was completed at the Institute for Mathematical Sciences, National University of Singapore (Summer 2014).  The first two authors thank them both for their hospitality, and acknowledge support from U.S. National Science Foundation grants DMS 1107452, 1107263, 1107367 ``RNMS: GEometric structures And Representation varieties" (the GEAR Network).  The second author also thanks the University of Chicago for hosting him for a short stay in November 2009 where some of the questions addressed in this paper were suggested during the Geometry/Topology Seminar.  We also thank Oscar Garc\'ia-Prada, Bill Goldman, and Juan Souto for helpful discussions, and we thank Geoffroy Horel and Oscar Randall-Williams for helpful discussions on MathOverflow \cite{MO-Un}.  Lastly, we thank an anonymous referee.

\section{Free group character varieties}\label{charvar-sec}
In this section we collect some general facts about free group character varieties that will be used in later sections.  The reader may wish to skip this section and refer back to it as needed in the sequel.

Let $\F_r=\langle e_1,...,e_r\rangle$ be the free group on $r$ generators $e_1,...,e_r$, and let $G$ be either a connected reductive $\C$--group, or a connected compact Lie group\footnote{Given \cite{CFLO} we expect some of our results to generalize to the setting of non-compact real reductive groups.  See Conjecture \ref{conj} for example.}.   In this setting $G$ is always a linear algebraic group.  As in the introduction $DG = [G,G]$ is the derived subgroup.  Let $PG$ be the quotient of $G$ by its center $Z(G)$.    Then 
\begin{equation} \label{decomp}G \isom (DG \cross T)/F, \end{equation}
where $T \leqs Z(G)$ is a central torus in $G$ and the finite group $F = DG\cap T$ 
acts diagonally on $DG \cross T$. As shown in \cite{Lawton-Ramras}, this implies 
\begin{equation}\label{diag}\XC{r}(G)\cong \XC{r}(DG)\times_{F^r}T^r,\end{equation}
where the notation on the right indicates that we mod out the diagonal action of $F^r$ on 
$\XC{r}(DG) \cross T^r$, with $F^r$ acting by coordinate-wise multiplication on both $\XC{r}(DG)$ and $T^r$.

When $G$ is a reductive $\C$--group, we call a representation $\rho\in\hom(\F_r, G)$ \e{irreducible} if the image $\rho(\F_r)$ is not contained in a proper parabolic subgroup of $G$.  
By \cite[Prop. 15]{Si4}, a representation $\rho \in \hom(\F_r, G)$ is irreducible if and only if its stabilizer in $PG$ is finite, which means that $\rho$ is stable in the affine GIT sense (see \cite[Prop. 5.11]{CaFl}).  We denote the set of irreducible representations by $\hom(\F_r,G)^{irr} \subset \hom(\F_r, G)$, and its complement $\hom(\F_r,G)^{red}$ is the set of reducible representations\footnote{We note that when $G$ is Abelian, $\hom(\F_r,G)^{irr} = \hom(\F_r, G)$.}.  As shown in \cite{Si4},  $\hom(\F_r,G)^{red}$ is an algebraic subset of $\hom(\F_r, G)$.  If $K$ is a connected compact Lie group with complexification $G$, then $\XC{r} (K)$ naturally embeds in $\XC{r} (G)$ by \cite[Theorem 4.3]{FlLa-quiver}.
We define the set of irreducible representations to be $\hom(\F_r, K)^{irr} = \hom(\F_r, K) \cap \hom(\F_r, G)^{irr}$, and similarly for reducibles; since $K$ is a real algebraic subset of $G$, we see that $\hom(\F_r, K)^{red}$ is a real algebraic subset of $\hom(\F_r, K)$.

If $G$ is a reductive $\C$--group, then $\rho$ is {\it completely reducible} if for every proper parabolic $P$ containing $\rho(\F_r)$, there is a Levi subgroup $L<P$ with $\rho(\F_r)< L$.  (Note that irreducible representations are, vacuously, completely reducible.)
By \cite{Si4}, a representation $\rho \co \F_r \to G$ is completely reducible if and only if it is polystable (has a closed adjoint orbit), so the set $\hom(\F_r,G)^{irr} \subset \hom(\F_r, G)$ of irreducibles lies inside the set $\hom(\F_r, G)^*$ of closed orbits.  
Let $\XC{r}(G)^{irr} \subset \XC{r} (G)$ denote the conjugation quotient of $\hom(\F_r,G)^{irr}$.  Also, let $\XC{r}(G)^{red}=\XC{r}(G)- \XC{r}(G)^{irr}$ be the {\it reducible locus}; it is a subvariety (see \cite{Si4}).   Since every completely reducible representation has a closed orbit, we conclude that 
$$\XC{r}(G)^{red}\cong (\hom(\F_r,G)^{red}\cap \hom(\F_r, G)^*)/G.$$  

If $G$ is either a connected reductive $\C$--group or a connected compact Lie group, then (following  \cite{JM}) we define the {\it good locus} to be the subspace
$$\hom(\F_r,G)^{good}\subset\hom(\F_r,G)^{irr}$$
of representations whose $PG$--stabilizer is trivial, and we define $\XC{r}(G)^{good} \subset \XC{r} (G)$ to be the quotient $\hom(\F_r,G)^{good}/G$.

When $G$ is a reductive $\C$--group, $\XC{r}(G)$ is an algebraic variety and we define $\XC{r}(G)^{sm}:=\XC{r}(G)-\XC{r}(G)^{sing}$, where $\XC{r}(G)^{sing}$ is the subvariety of singular points.   
If $K < G$ is a maximal compact subgroup, we define $\XC{r} (K)^{sing} = \XC{r} (G)^{sing} \cap \XC{r} (K)$, and similarly for $\XC{r} (K)^{sm}$.  
As shown in \cite{FlLa2} and more generally in \cite{Si4}, $\XC{r}(G)^{good}$ is a submanifold of the complex manifold $\XC{r}(G)^{sm}$ and $\X_r(G)^{irr}$ is an orbifold in $\XC{r}(G)$.  Moreover, both $\hom(\F_r,G)^{irr}$ and $\hom(\F_r,G)^{good}$ are Zariski open (non-empty for $r\geqs 2$) subspaces of $G^r$, and as such are smooth manifolds too.  By \cite{BiLa}, these latter spaces are in fact bundles over the former:

\begin{lem}[Lemma 2.2 in \cite{BiLa}]\label{lem22BL}
Let $G$ be a connected reductive $\C$--group $($respectively a connected compact Lie group$)$.
Then $$\hom(\F_r,G)^{good}\to \XC{r}(G)^{good}$$
is a principal $PG$--bundle, and $$\hom(\F_r,G)^{irr}\to \XC{r}(G)^{irr}$$ is a $PG$--orbibundle in the \'etale topology $($respectively the usual topology$)$.
\end{lem}

We note that \cite[Lemma 2.2]{BiLa} refers only to the reductive case.  The argument in the compact case is the same; for the orbibundle statement one uses Proposition~\ref{irred-stab} below.

The isomorphism (\ref{decomp}) implies that if $\rho\in \hom(\F_r, DG)$, the adjoint orbits $\{g\rho g^{-1} \,:\, g\in DG\}$ and $\{g\rho g^{-1} \,:\, g\in G\}$ coincide, and it follows that 
there are natural inclusions $\XC{r} (DG) \injects \XC{r} (G)$ and $\XC{r} (DG)^{irr} \injects \XC{r} (G)^{irr}$.

\subsection{Stabilizers and irreducible representations in compact groups}

In this section, $K$ will denote a  compact, connected Lie group with complexification $G = K_\C$.  We will regard $K$ as a subgroup of $G$.  Recall that there exists a Cartan decomposition 
\begin{equation}\label{CD}K\cross \exp (\mathfrak{p}) \srm{\isom} G,\end{equation}
 $(k, \exp(p)) \mapsto k \exp(p)$, where $\mathfrak{p} \leqs \mathrm{Lie}(G)$ is the $(-1)$--eigenspace of a Cartan involution $\theta$ on $\mathrm{Lie} (G)$.

\begin{lem} \label{centralizer-SDR} Let $G$ be a connected reductive $\C$--group with maximal compact subgroup $K \leqs G$.  Then for each subgroup $H \leqs K$,  the centralizer $C_G (H)$ deformation retracts to $C_K (H)$. 
Moreover, the inclusion $C_K (H) \to C_G(H)$ induces a homotopy equivalence 
$$C_K (H)/Z(K) \srm{\simeq} C_G(H)/Z(G).$$
\end{lem}
\begin{proof}
 We claim that the map
$C_G (H) \cross [0,1]\to C_G (H)$, $(k \exp(p), t) \mapsto k \exp((1-t)p)$, is a deformation retraction onto $C_K (H)$.  
We need to prove that for each $t\in [0,1]$  and each $k \exp(p)\in C_G (H)$, we have 
$k \exp(tp)\in C_G (H)$ and $k = k\exp(0) \in C_K (H)$.

First we check that if  $k \exp(p)\in C_G (H)$ for some $k\in K$, $p\in \mathfrak{p}$, then $k \in C_K (H)$.  Say $h\in H$.  We have
$$k \exp(p) h \exp(-p) k^{-1} = h,$$
so 
$$k^{-1} h k = \exp(p) h \exp(-p) h^{-1} h = \exp (p) \exp( -\mathrm{ad}(h) p) h.$$
Hence
$$k^{-1} h k h^{-1} \exp( \mathrm{ad}(h) p) =  \exp (p).$$
But $k^{-1} h k h^{-1}\in K$ and $\exp( \mathrm{ad}(h) p) \in \mathfrak{p}$ (because the adjoint action of $K$ fixes the $(-1)$--eigenspace $\mathfrak{p}$ of $\theta$), so by bijectivity of the Cartan decomposition map (\ref{CD}), we must have $k^{-1} h k h^{-1} = 1$ and $\exp(\mathrm{ad}(h) p) = \exp(p)$.  Since $h\in H$ was arbitrary, we conclude that $k\in C_K (H)$.

Next, we must show that for each $t\in [0,1]$ and each $k\in K$, $p\in \mathfrak{p}$ with $k\exp(p) \in C_G (H)$, 
 we have $k \exp(tp)\in C_G (H)$.  Fix $t\in [0,1]$ and $h\in H$.
 The above computation shows that $k\in C_K (H)$ and that
$\exp(\mathrm{ad}(h) p) = \exp(p)$.  Since the exponential map is injective on $\mathfrak{p}$, we have $\mathrm{ad} (h) p = p$, and hence $\exp(t\mathrm{ad} (h)p) = \exp(tp)$.
 Following the above computation in reverse, we have 
 $$\exp(tp) = k^{-1} h k h^{-1} \exp(t\mathrm{ad} (h)p)$$
 so
 $$ k^{-1} h k = \exp(tp) \exp(-\mathrm{ad} (h) tp) h = \exp(tp) h \exp(- tp)$$
 and hence
 $$h = k  \exp(tp) h \exp(- tp) k^{-1},$$
 showing that $k  \exp(tp) \in C_G (H)$, as desired.
It follows that the inclusion $C_K (H) \injects C_G (H)$ is a homotopy equivalence.
To see that the natural map $f\co C_K (H)/Z(K) \to C_G(H)/Z(G)$ is a homotopy equivalence, 
consider the commutative diagram,
$$\xymatrix{ 	Z(K) \ar@{=}[r] \ar[d] & Z(K)\ar[d]\\
			C_K (H) \ar[r]^\heq \ar[d] & C_G (H) \ar[d]\\
			C_K (H)/Z(K) \ar[r] & C_G(H)/Z(K) \ar[r] &C_G (H)/Z(G),}$$
in which the composite of the bottom row is exactly $f$.						
The first two columns in this diagram are (surjective) fibration sequences, so we conclude that the first map in the bottom row is a homotopy equivalence.  The kernel of the homomorphism $C_G(H)/Z(K) \to C_G (H)/Z(G)$ is $Z(G)/Z(K)$, which is contractible since $Z(K)$ is the maximal compact subgroup of $Z(G)$.  Therefore $C_G(H)/Z(K) \to C_G (H)/Z(G)$ is a homotopy equivalence.
\end{proof}

Note that the $K$--orbit of a representation $\rho\co \F_r \to K$ is always closed, since $K$ is compact.  In fact, the $G$--orbit of such a representation is also closed, as shown in \cite{FlLa}; that is, $K$--valued representations in $\hom(\F_r, G)$ are always polystable.
By \cite[Theorem 30]{Si4}, polystability is equivalent to complete reducibility, which yields the following result.

\begin{lem}\label{K-red}  If  $\rho\co \F_r \to K$ has image contained in a parabolic subgroup $P<G$, then there exists a Levi subgroup $L< P$ with $\rho (\F_r) < L$.
Moreover, $\rho$ is reducible if and only if there exists a parabolic subgroup $P< G$ and a Levi subgroup $L < P$ such that $\rho (\F_r) < L$.  
\end{lem}

\begin{lem}\label{reductive-stab}
If $\rho\co \F_r \to G$ is polystable, then $Stab_G (\rho)$ is a $($not necessarily connected$)$ reductive $\C$--group.  In particular, for every representation $\rho\co \F_r \to K$, $Stab_G (\rho)$ is a reductive $\C$--group.
\end{lem}
\begin{proof} By Matsushima's Theorem, a subgroup $H$ of a reductive $\C$--group $G$ is itself reductive if and only if $G/H$ is an affine algebraic variety.  Since $G/Stab_G (\rho) \isom O_\rho$, where $O_\rho \subset G^r$ is the adjoint orbit of $\rho$, we see that $O_\rho$ is an affine variety whenever it is Zariski closed.  However, as observed in \cite[Section 2]{FlLa3}, since $O_\rho$ is the image of an algebraic map, it is constructible, so $O_\rho$ is closed in the analytic topology on $G^r$ if and only if it is Zariski closed.
\end{proof}

\begin{prop}\label{irred-stab} A representation $\rho\co \F_r \to K$ is irreducible if and only if the $PK$--stabilizer of $\rho$ is finite.  
\end{prop} 
\begin{proof} Irreducibility means that the $PG$--stabilizer $Stab_G (\rho)/Z(G)$ is finite (see \cite{Si4}, and \cite{CaFl}).  The kernel of the composition
$$Stab_K (\rho) \maps Stab_G (\rho) \maps Stab_G (\rho)/Z(G),$$
is precisely $Stab_K (\rho) \cap Z(G) = Z(K)$, so
$Stab_K (\rho)/ Z(K)$ is a subgroup of the finite group $Stab_G (\rho)/Z(G)$ (in fact, the two are equal by Lemma~\ref{centralizer-SDR}).

In the other direction, say $Stab_K (\rho)/ Z(K)$ is finite.   By Lemma~\ref{centralizer-SDR}, 
$$Stab_G (\rho)/Z(G) = C_G (\rho(\F_r))/Z(G) \heq C_K (\rho(\F_r))/Z(K) = Stab_K (\rho)/ Z(K),$$
 which implies that the identity component $C$ of $Stab_G (\rho)/Z(G)$ is contractible.  By Lemma~\ref{reductive-stab}, $Stab_G (\rho)$ is a reductive $\C$--group, and the same is true of $C$ (since connected reductive $\C$--groups are characterized as affine algebraic groups over $\C$ whose radical is a torus).
Now, $C$ deformation retracts to its maximal compact subgroup $H < C$, which is a closed manifold.  But $H$ is contractible, so  $H = \{1\}$, and since $C$ is the complexification of $H$, we have $C = \{1\}$ as well.  Hence $Stab_G (\rho)/Z(G)$ is finite, and by  \cite{Si4}, we conclude that $\rho$ is irreducible.
\end{proof}

\begin{rem}
Note that the last three results are valid for representations of any finitely generated group $\Gamma$ replacing $\F_r$, since the respective proofs work in this more general case. 
\end{rem}

\subsection{The reducible locus}

In this section we study the reducible locus $\hom(\F_r, G)^{red}$.  By \cite[Proposition 27]{Si4}, this is an algebraic subset of $\hom(\F_r, G)$  and we begin with an analysis of its irreducible components.

\begin{lem}\label{H_P} Let $G$ be a connected reductive $\C$--group.
For a proper parabolic $P< G$, define $H_P:=\cup_{g\in G}\hom(\F_r,gPg^{-1})$.  Then $H_P$
and $H_P\aq G$ are irreducible algebraic sets, and for maximal $P$, these sets are exactly the irreducible components of $\hom(\F_r, G)^{red}$ and $\XC{r}(G)^{red}$ $($respectively$)$. 
\end{lem}
\begin{proof}
First note that $H_P$ is an algebraic set itself, again by \cite[Proposition 27]{Si4}, and moreover the map $\hom(\F_r,P)\times G\to H_P$ given by $(\rho, g)\mapsto g\rho g^{-1}$ is a surjective algebraic map from an irreducible variety to an algebraic set.  In general, a surjective morphism has the property that if the inverse image of an algebraic set is irreducible, then the original set was too.  Therefore, we conclude that $H_P$ is irreducible.  This implies $H_P\aq G$ is irreducible.  
Conversely, the reducible locus $\hom(\F_r, G)^{red}$ is the union of the irreducible subsets $H_P$, so its irreducible components are the maximal elements of the collection $\{H_P\}_P$ (ordered by inclusion), and these correspond to the maximal parabolics.  The situation for $\XC{r}(G)^{red}$ is analogous.
\end{proof}

The next lemma is a technical result that allows us to compute the codimension of the reducible locus.

\begin{lem}\label{finiteKconjugates}
Let $G$ be a connected reductive $\C$--group, and fix a maximal compact subgroup $K < G$.  Then 
for every parabolic subgroup $P \leqs G$, we have $KP = G$.
Consequently, the set
$$\{ H \leqs K \,:\, H = K\cap P \textrm{ for some parabolic subgroup } P \leqs G\}$$
contains only finitely many $K$--conjugacy classes of subgroups.
\end{lem}
\begin{proof}
The Iwasawa decomposition of $DG$ implies that $DG = (DK) B'$ for some Borel subgroup  $B' < DG$ (see \cite{Bump}). 
Let $B < G$ be a Borel subgroup containing $B'$.  We have
$$G = (DG) (Z(G)) = (DK) B' (Z(G)) = (DK) B$$ 
(since $B', Z(G) < B$), so $G=KB$.  
Now let $P < G$ be parabolic.  Recall that $P$ contains a Borel subgroup of $G$, and all Borel subgroups of $G$ are conjugate, so $gBg^{-1} < P$ for some $g\in G$.  Since $G=KB$, we can write $g = kb$ for some $k\in K$, $b\in B$, and now 
$$gBg^{-1} = kbBb^{-1}k^{-1} = kBk^{-1}.$$  
Hence
$$KP \supset KgBg^{-1} = KkBk^{-1} = KBk^{-1} = G,$$
as claimed.

It is well known that there are only finitely many $G$--conjugacy classes of parabolic subgroups in $G$ (see \cite{Borel}).  
In fact, if $P, Q < G$ are parabolic and conjugate in $G$, then we have $gPg^{-1} = Q$ for some $g\in G$.  Since $G = KP$, we can write $g = kp$ for some $k\in K$ and some $p\in P$, and we have
$$Q = gPg^{-1} = kpPp^{-1} k^{-1} = kPk^{-1}.$$
Hence there are only finitely many $K$--conjugacy classes of parabolic subgroups in $G$, and the final statement of the lemma follows because for any $k\in K$ we have $k(P\cap K)k^{-1}=kPk^{-1}\cap K$.
\end{proof}

Our main results will rely crucially on the fact that the reducible locus has sufficiently large codimension in the representation space.
Recall that the rank $\mathrm{Rank}(G)$ of a complex (respectively, compact) reductive Lie group $G$ is the dimension over $\C$ (respectively, over $\R$) of a maximal torus.

\begin{thm}\label{gencodim}
Let $G$ be a connected reductive $\C$--group, or a connected compact Lie group.  Then $\hom(\F_r,G)^{red}$ has real codimension at least $4$ when $r\geqs 3$, and also when $r\geqs2$ and $\mathrm{Rank}(DG)\geqs2$.
\end{thm}

\begin{proof}
We first handle the complex case.  In general, $\hom(\F_r, G)^{red} =\cup_P \hom(\F_r,P)$ where $P$ runs through all proper parabolic subgroups of $G$.  
Up to conjugation there are only finitely many such $P$'s.  So the dimension is determined by one of the maximal dimensional proper parabolics. 
Call one of these $P_{max}$.  Recall that $H_{P_{max}}$ is the image of the mapping $q\co G\times \hom(\F_r,P_{max})\to \hom(\F_r,G)$ given by $(g,\rho)\mapsto g\rho g^{-1}$, and $H_P$ is an irreducible algebraic set by Lemma \ref{H_P}.  The morphism $q$ factors through the quotient set $(G\times \hom(\F_r,P_{max}))/P_{max}$, where the action of $P_{max}$ on $G\times \hom(\F_r,P)$ is given by $p\cdot(g,\rho)=(gp^{-1},p\rho p^{-1})$.  Therefore, every fiber of $q$ contains a copy of $P_{max}$ (up to isomorphism); this implies $\dim_\C \mathcal{F} \geqs \dim_\C P$ for every fiber $\mathcal{F}$ of $q$.\footnote{The fiber is equal to $P_{max}$ whenever $\rho(\F_r)$ is a Zariski dense subgroup of $P_{max}$, since the normalizer of $P_{max}$ in $G$ is $P_{max}$ itself.
And $\rho(\F_r)$ is generically dense when $r\geqs \max\{2, \dim P/L\}$ where $L$ is a Levi subgroup, since two elements generically generate a dense subgroup in a compact Lie group, and consequently two elements generically generate a Zariski dense subgroup in a reductive group.}  By the Fiber Dimension Theorem (see \cite{Sh1}), we conclude there exists a (generic) fiber $\mathcal{F}$ such that:\begin{eqnarray*}\dim_\C \left(\hom(\F_r,G)^{red}\right)&=&\dim_\C \left(H_{P_{max}}\right)\\&=&r(\dim_\C P_{max})+\dim_\C G -\dim_\C \mathcal{F}\\&\leqs& (r-1)\dim_\C(P_{max})+\dim_\C G.\end{eqnarray*}  Therefore, 
 \begin{eqnarray*}\codim_\C \left(\hom(\F_r,G)^{red} \right)\\ \geqs& r(\dim_\C G) - \left[(r-1)\dim_\C\left(P_{max}\right)+\dim_\C (G)\right] \\ =& (r-1)\left(\dim_\C (G)-\dim_\C \left(P_{max}\right)\right).\hspace{.8in}
 \end{eqnarray*}  So for $r\geqs 3$, we have at least complex codimension 2 (and hence real codimension at least 4) in general.  
 
Now suppose $\mathrm{Rank}(DG)\geqs2$ and let $P < G$ be a proper parabolic subgroup.  Then the codimension of $P$ in $G$ is the dimension of $G/P$, which is isomorphic to the flag variety $DG/(P\cap DG)$.  The flag variety has Schubert cells (which are CW cells) given by $B$--orbits of the left translation action of $DG$ on $DG/(P\cap DG)$, where $B\subset P$ is a Borel subgroup of $DG$.  These cells necessarily contain a maximal torus $T$ of $DG$ where $T\subset B$ (by the Bruhat decomposition; see \cite{Bump}).  Since $\mathrm{Rank}(DG)\geqs2$, we have $\dim_\C (T)\geqs2$ and so the Schubert cells and consequently the flag variety have dimension at least $2$.  We conclude the complex codimension of $P$ in $G$ is at least $2$, which implies (by the above inequalities) that the complex codimension of $\hom(\F_r,G)^{red}$ is at least 2 when $r\geqs 2$ (and again the real codimension is at least 4).

Next, consider the case in which $G$ is compact.   Then 
$$\hom(\F_r, G)^{red} =\bigcup_{P} \hom(\F_r,P\cap G),$$ 
where $P$ runs through all proper parabolic subgroups of the complexification $G_\C$.  Lemma \ref{finiteKconjugates} implies that the dimension is determined by one of the maximal dimensional proper parabolics in $G_\C$; again call it $P_{max}$.  A similar argument as above, where we use Hardt's Theorem (see \cite{Hardt}) instead of the Fiber Dimension Theorem, shows that $$\codim_\R \left(\hom(\F_r,G)^{red} \right) \geqs (r-1)\left(\dim_\R (G)-\dim_\R \left(P_{max}\cap G\right)\right).$$ 

By Lemma \ref{finiteKconjugates} we have $G_\C/P = GP/P  \cong G/(P\cap G)$  for every parabolic $P\leqs G_\C$.  Therefore, since the complex dimension of $G_\C/P$ is at least 1, the real codimension of $P\cap G$ in $G$ is at least 2.  And as shown above, when $\mathrm{Rank}(DG) = \mathrm{Rank}(D(G_\C)) \geqs2$, the real codimension of $P\cap G$ in $G$ is at least 4.

Hence when $r\geqs3$, and also when $r\geqs2$ and $\mathrm{Rank}(DG)\geqs2$, the real codimension of $\hom(\F_r, G)^{red}$ is at least 4.
\end{proof}

\subsection{Global description of loci}

We now show how these various loci relate to each other.   In this section, unless otherwise noted, $G$ will denote either a connected compact Lie group, or a connected reductive $\C$--group. 

The center $Z(G)^r < G^r$ acts on $\hom(\F_r, G)$ by setting $\rho \cdot z$ to be the representation $(\rho\cdot z) (e_i) = \rho(e_i) z_i$, where $z = (z_1, \ldots, z_r)$.  This action descends to an action on $\XC{r}(G)$ since it commutes with the adjoint action of $G$ and preserves the polystable locus in $\hom(\F_r,G)$.
Consider a property P such that if $\rho\in \hom(\F_r, G)$ satisfies P, then  for all $g\in G$ and $z\in Z(G)^r$, both $g\rho g^{-1}$ and $\rho\cdot z$ satisfy P.  
Let $\hom(\F_r, G)^\mathrm{P} = \{\rho \,:\, \rho \textrm{ satisfies P}\}$ and let $\XC{r}(G)^\mathrm{P} \subset \XC{r} (G)$ denote the image of $\hom(\F_r, G)^\mathrm{P}$ under the quotient map
$\hom(\F_r, G) \to \XC{r} (G)$.  Let $\XC{r}(DG)^\mathrm{P} = \XC{r}(DG)\cap \XC{r}(G)^\mathrm{P}$.  

The following result extends the free group case of  \cite[Lemma 2.3]{BiLaRa}, since 
when P is vacuous, we have $\XC{r}(G)^\mathrm{P} = \XC{r}(G)$ and  $\XC{r}(DG)^\mathrm{P} = \XC{r}(DG)$.

\begin{thm}\label{P-theorem} 
Let $T$ be a maximal central torus in $G$, and set $F=T\cap DG$.
Then  $\XC{r}(G)^\mathrm{P}\cong \XC{r}(DG)^\mathrm{P}\times_{F^r} T^r$ $($in the notation from $(\ref{diag})$$)$,
and 
\begin{equation} \label{P-fibn} \XC{r} (DG)^\mathrm{P}\maps \XC{r} (G)^\mathrm{P} \maps \XC{r} (G/DG)\end{equation}
is a Serre fibration sequence.
\end{thm}

\begin{proof}
We have $G \isom (DG \cross T)/F$ as in (\ref{decomp}).
The product fibration sequence $$\XC{r} (DG)^\mathrm{P}\to\XC{r} (DG)^\mathrm{P} \cross T^r \to T^r$$ admits a free action of $F^r \leqs Z(G)^r$ on the base and total space, making the projection equivariant: on the total space, the action is $f\cdot [[\rho], t] = [ [\rho]\cdot f^{-1}, f t]$ and on the base it is simply left multiplication.  By~\cite[Corollary A.2]{BiLaRa}, there is an induced fibration sequence after passing to the quotient spaces, and the fiber is homeomorphic to $\XC{r} (DG)^\mathrm{P}$:  $$\XC{r} (DG)^\mathrm{P}\maps \XC{r} (DG)^\mathrm{P} \cross_{F^r} T^r \to \XC{r} (G/DG) \isom T^r/F^r.$$

We claim that the image of the map 
\begin{equation}\label{act-map}\XC{r} (DG)^\mathrm{P} \cross_{F^r} T^r \maps\XC{r} (G),\end{equation} 
given by $[[\rho],t]\mapsto [\rho\cdot t]$, lies in $\XC{r} (G)^\mathrm{P}$.  This follows since if $\rho$ satisfies P, then by assumption so does $\rho\cdot z$ for any $z\in Z(G)^r$, and since $T\subset Z(G)$ we have $t\in Z(G)^r$.

It remains to check that every point in $\XC{r} (G)^\mathrm{P}$ is in the image of (\ref{act-map}).  Say $[\rho]\in \XC{r} (G)^\mathrm{P}$, and write $\rho(e_i) = \rho'(e_i) t_i$ for some $\rho'(e_i)\in DG$ and $t_i\in T$.  Then, setting $t = (t_1, \ldots, t_r)$, we have $[[\rho'], t] \mapsto [\rho]$, and $[\rho']$ satisfies P since $\rho$ satisfies P and $[\rho'] = [\rho]\cdot t^{-1}$.

This completes the proof that $\XC{r}(G)^\mathrm{P}\cong \XC{r}(DG)^\mathrm{P}\times_{F^r} T^r$, and also shows that (\ref{P-fibn}) is a Serre fibration.
\end{proof}

\begin{cor}\label{P-cor}  
$\XC{r}(G)\cong \XC{r}(DG)\times_{F^r} T^r$, $\XC{r}(G)^{irr}\cong \XC{r}(DG)^{irr}\times_{F^r} T^r$, $\XC{r}(G)^{red}\cong \XC{r}(DG)^{red}\times_{F^r} T^r$, and $\XC{r}(G)^{good}\cong \XC{r}(DG)^{good}\times_{F^r} T^r.$
\end{cor}
\begin{proof} By Theorem~\ref{P-theorem}, we just need to check that each of these properties is invariant under conjugation and under multiplication by central elements.  But the above properties are characterized in terms of stabilizers (see Proposition \ref{irred-stab} in particular), and these operations do not change the stabilizers.  Note that for these properties, $\XC{r} (DG)^\mathrm{P}$ (as defined above) agrees with the intrinsic notion.
\end{proof}

We next generalize \cite[Corollary 2.6, Corollary 2.7]{FlLa2} and thereby reduce the study of singularities in $\XC{r}(G)$ for a general reductive $\C$--group $G$ to the semisimple case.  Let $X^{sing}$ denote the singular locus of an algebraic variety $X$ 
and let $X^{sm}=X-X^{sing}$ be the smooth locus.

\begin{prop}\label{etale} 
Let $G$ be a connected, reductive $\C$--group, and let $T$ be a maximal central torus in $G$.  Then $\XC{r}(G)$ and $\XC{r}(DG)\times T^r$ are \'etale equivalent.  Consequently, $\XC{r}(G)^{sing}\cong \XC{r}(DG)^{sing}\times_{F^r} T^r$
and $\XC{r}(G)^{sm}\cong \XC{r}(DG)^{sm}\times_{F^r} T^r$.
\end{prop}

\begin{proof}
By Corollary~\ref{P-cor}, we have $\XC{r}(G)\cong \XC{r}(DG)\times_{F^r} T^r$, where $F^r$ is acting freely.  Since $DG^r\times T^r$ and $G^r$ are smooth, they are normal.  This implies (see \cite{DrJ}) that $(DG^r\times T^r)\aq DG=\XC{r}(DG)\times T^r$ is also normal.   However, the GIT projection 
$$\XC{r}(DG)\times T^r\to\XC{r}(DG)\times_{F^r}T^r$$ is then \'etale because $F^r$ is 
finite and acts freely (see \cite{DrJ}), which completes the proof that $\XC{r}(G)$ and $\XC{r}(DG)\times T^r$ are \'etale equivalent.  

Next, since \'etale maps determine local isomorphisms (at both singularities and smooth points), we immediately conclude (since $T^r$ is smooth) that $\XC{r}(G)^{sing}\cong \XC{r}(DG)^{sing}\times_{F^r} T^r$ and $\XC{r}(G)^{sm}\cong \XC{r}(DG)^{sm}\times_{F^r} T^r$, as required.
\end{proof}

\subsection{Local description of loci}

We now note some local properties of $\XC{r}(G)$ that are important in the study of its global topology.
 
An algebraic variety is irreducible if and only if there exists a path connected Zariski open dense smooth subset (for a proof, see \cite[Theorem 8.4]{CoMa}). Since $\XC{r}(G)$ is normal this fact is also local, in the following sense.

When $G$ is a connected reductive $\C$--group, the variety $\XC{r} (G)$ is normal (since it is a GIT quotient of the smooth, connected variety $G^r$ \cite{Do, Sh1}).  Mumford's topological version of Zariski's main theorem \cite[III.9]{Mu} shows that for every point $x$ in a normal, affine variety $X$, and every neighborhood $V$ of $x$ in $X$ (in the Euclidean topology), there exists a neighborhood $V' \subset V$ such that $V\cap X^{sm}$ is path connected.

We now show that in the compact case, there are contractible neighborhoods around 
certain reducibles that remain simply connected after removing the reducible locus.  When $K = \Um{n}$, this lemma actually applies to \e{all} reducibles, since by \cite[Lemma 4.3]{Ramras2}, every subgroup of $ \Um{n}$ has connected centralizer.

\begin{lem}\label{locallysimplyconnected} Let $K$ be a compact Lie group.  Suppose that either $r\geqs3$, or $\mathrm{Rank}(DK)\geqs 2$ and $r\geqs 2$.  Let $\rho \in \hom(\F_r, K)$ be a representation such that $Stab (\rho)/Z(K)$ is path connected, where $Stab (\rho) \leqs K$ is the stabilizer of $\rho$ under the adjoint action.
Then for every neighborhood $U \subset \XC{r} (K)$ containing $[\rho]$, there exists a contractible neighborhood $V \subset U$, with $[\rho]\in V$, such that $V^{irr} := V\cap \X_r (K)^{irr}$ is simply connected.
\end{lem}

\begin{proof}  Let $K$ act by conjugation on $\hom(\F_r, K) \cong K^r$, and choose a $K$--invariant metric on the tangent bundle $T(\hom(\F_r, K))$.  Let 
$$O_\rho\subset \hom(\F_r, K)$$ 
denote the conjugation orbit of $\rho$, and let $T_\rho (O_\rho) \subset T_\rho (\hom(\F_r, K))$ denote the tangent space to the orbit at $\rho$.
Let 
$$K\cross_{Stab (\rho)} D_\epsilon = (K \cross D_\epsilon)/Stab (\rho),$$ 
 where $D_\epsilon$ is the ball of radius $\epsilon$ in the orthogonal complement of $T_\rho (O_\rho)$ inside $T_\rho (\hom(\F_r, K))$
and the action is given by
$s \cdot (k, x) = (ks^{-1}, \mathrm{ad}(s) x)$.  Note that $K$ acts on $K\cross_{Stab (\rho)} D_\epsilon$ via multiplication on the left.
By the Slice Theorem~\cite[Theorem B.24]{GGK}, there exists an $\epsilon > 0$ and a $K$--equivariant diffeomorphism 
$$\phi \co K\cross_{Stab (\rho)} D_\epsilon \srm{\isom} \wt{V}$$
onto a neighborhood $\wt{V}$ of $\rho$ in $\hom(\F_r, K)$.  
By $K$--equivariance, this diffeomorphism descends to a homeomorphism from $(K\cross_{Stab (\rho)} D_\epsilon)/K \isom D_\epsilon/Stab (\rho)$ onto a neighborhood $V$ of $[\rho]$ in $\XC{r} (K)$.  
By shrinking $\epsilon$ if necessary, we may assume that $V\subset U$.  

We claim that $V$ satisfies the desired properties.  First, $Stab (\rho)$ is a compact Lie group, so $D_\epsilon/Stab (\rho)$ is contractible by the Conner Conjecture (Oliver's Theorem)~\cite{Oliver}.  Next, we must show that $V^{irr}$ is simply connected.  Let $D_\epsilon^{red} = \{x\in D_\epsilon \,:\, \phi ([1_K, x]) \in \hom(\F_r, K)^{red}\}$, where $1_K$ is the identity element in $K$.

Define
$$(K\cross_{Stab(\rho)} D_\epsilon)^{red} := \{[k,x] \in K\cross_{Stab(\rho)} D_\epsilon \,:\, \phi ([k, x]) \in \hom(\F_r, K)^{red}\}.$$  
We claim that the (real) codimension\footnote{Since $\hom(\F_r, K)^{red}$ is an algebraic subset of $\hom(\F_r, K)$, it can be written as a union of locally closed submanifolds (see \cite{BCR}), and hence the same is true of $D_\epsilon^{red}$ and $(K\cross_{Stab(\rho)} D_\epsilon)^{red}$; the dimensions of these spaces then refer to the maximum dimension of one of these submanifolds.} of $D_\epsilon^{red}$ in $D_\epsilon$ is the same as the codimension of 
$(K\cross_{Stab(\rho)} D_\epsilon)^{red}$ in $(K\cross_{Stab(\rho)} D_\epsilon)$.
Since $\hom(\F_r, K)^{red}$ is a $K$--invariant subset of $\hom(\F_r, K)$ and $\phi$ is $K$--equivariant, we have
$$(K\cross_{Stab(\rho)} D_\epsilon)^{red}  = K\cross_{Stab(\rho)} D_\epsilon^{red},$$
and since $Stab (\rho)$ acts freely on $K\cross D_\epsilon$, we have
 $$\dim_\bbR (K\cross_{Stab(\rho)} D_\epsilon) - \dim_\bbR (K\cross_{Stab(\rho)} D_\epsilon)^{red}=$$ $$\dim_\bbR (K) + \dim_\bbR (D_\epsilon) -\dim_\bbR (Stab (\rho))
- \left[\dim_\bbR (K) + \dim_\bbR (D_\epsilon)^{red} - \dim_\bbR (Stab (\rho))\right]$$
$$ = \dim_\bbR  D_\epsilon -  \dim_\bbR D_\epsilon^{red}.$$
Now, $\dim_\bbR (K\cross_{Stab(\rho)} D_\epsilon) = \dim_\bbR \hom(\F_r, K)$, while 
$$\dim_\bbR (K\cross_{Stab(\rho)} D_\epsilon)^{red} \leqs \dim_\bbR  \hom(\F_r, K)^{red},$$ 
so
the codimension of $(K\cross_{Stab(\rho)} D_\epsilon)^{red}$ in $(K\cross_{Stab(\rho)} D_\epsilon)$ is at least the codimension of $\hom(\F_r, K)^{red}$ in $\hom(\F_r, K)$, which is at least 4 by Theorem~\ref{gencodim}.  

We claim that $D_\epsilon^{irr} := D_\epsilon \setminus D_\epsilon^{red}$ is simply connected.  Given $x, y\in D_\epsilon^{irr}$, there exists a path between them in $D_\epsilon$, and by transversality\footnote{Our use of transversality in this context is analogous to \cite[Corollary 4.8]{Ramras}.  Here we are using the stratification of a real algebraic set by locally closed submanifolds, as in the previous footnote.}, there exists a path that avoids the subset $D_\epsilon^{red}$.  Similarly, every loop $\gamma$ in $D_\epsilon^{irr}$ is nullhomotopic in $D_\epsilon$, and applying transversality again shows that there is a nullhomotopy of $\gamma$ that avoids $D_\epsilon^{red}$.

The map $\pi\co D_\epsilon \to D_\epsilon/Stab(\rho) \srt{\isom} V$ restricts to a surjection $D_\epsilon^{irr} \to V^{irr}$, so $V^{irr}$ is path connected.  Moreover, since $Stab(\rho)$ is a compact Lie group, the  map 
$\pi$ satisfies path-lifting, as does its restriction $D_\epsilon^{irr} \to V^{irr}$.
The fibers of this map are quotients of $Stab (\rho)/Z(K)$, and hence are path connected by hypothesis.
It follows that the induced map $\pi_1 (D_\epsilon^{irr}) \to \pi_1 (V^{irr})$ is surjective~\cite[II.6]{Br}.  Since $D_\epsilon^{irr}$ is simply connected, we conclude that the same is true for $V^{irr}$.
\end{proof}

\begin{rem} 
We expect the above lemma to remain valid when $G$ is a reductive $\C$--group, as a consequence of the Luna Slice Theorem.  For example, let $G=\SLm{n}$.  Then in \cite{FlLa2}, the Luna Slice Theorem is used to show that at a generic singularity $($having stabilizer $\C^*$$)$, there exists a model neighborhood of the form $\C^N\times \mathcal{C}_\C(\cp^{m}\times\cp^{m})$ for appropriate values of $N$ and $m$, where $\mathcal{C}_\C(\cp^{m}\times\cp^{m})$ is the affine cone over $\cp^{m}\times\cp^{m}$.  The complement of the unique singularity in this affine cone $($the cone point$)$ is a bundle over $\cp^{m}$ with fiber $\C^{m+1}-0$.\footnote{Although it makes no difference to any result in \cite{FlLa2}, we take this opportunity to note that this sentence corrects Lemma 3.18 $(ii)$ in \cite{FlLa2}.}  Consequently, we see that the set of smooth points $($in this case equal to the irreducible points$)$ of this local $($contractible$)$ model is simply connected as long as $m>0$.  The above theorem achieves this in the compact case without knowing the specific homeomorphism type of the local model. 
\end{rem}

\section{Comparisons between homotopy groups}\label{fibration}

In this section, we compare the homotopy groups of the $G$--character varieties $\XC{r} (G):=\hom(\F_r,G)\aq G$, and of their various loci, to those of the corresponding $DG$, $K$, and $DK$ moduli spaces (see Theorem \ref{higherhomotopy} below). 

The following result allows us to ignore basepoints throughout our discussion of homotopy groups.

\begin{lem}\label{pi_0}
 Let $G$ be a connected reductive $\C$--group or a connected compact Lie group.
 Then both $\XC{r} (G)$ and $\XC{r}(G)^{irr}$ are path connected.
 \end{lem}
\begin{proof}First assume $G$ is complex.
Then $\XC{r} (G)$ is an irreducible algebraic set, hence path connected.  
The irreducible locus is Zariski open in $\hom(\F_r, G)$, and is non-empty (except when $r=1$ and $G$ is non-Abelian, in which case Jordan decomposition shows that $\XC{r}(G)^{irr}$ is empty and hence path connected).  Since $\hom(\F_r, G)$ is an irreducible algebraic set, every non-empty Zariski open subset of $\hom(\F_r, G)$ is path connected, and it follows that $\XC{r}(G)^{irr}$ is also path connected.  

Now assume $G$ is compact.   Since $G$ is path connected, so are $G^r$ and $\XC{r} (G)$.  Theorem~\ref{gencodim} together with transversality (as in Lemma~\ref{locallysimplyconnected}) establishes connectivity of $\XC{r}(G)^{irr}$ except in the cases $r = 1$, or $r=2$ and $\mathrm{Rank}(DG)\leqs 1$.  These low-rank cases can be addressed by hand using Corollary~\ref{P-cor}.
\end{proof}

To analyze the higher homotopy groups, we will use the following result.

\begin{thm}[\cite{FlLa}]\label{deformretract}Let $G$ be connected reductive $\C$--group and $K$ a maximal compact subgroup. Then $\XC{r}(K)$ is a strong 
deformation retract of $\XC{r}(G)$. In particular, these character varieties are homotopy equivalent.
\end{thm}

We come to the main result of this section, which reduces the study of higher homotopy groups for these character varieties, and their irreducible loci, to the semisimple case. 

\begin{thm}\label{higherhomotopy} Let $G$ be a connected reductive $\C$--group, with derived subgroup $DG = [G,G]$ and maximal compact subgroup $K$.  Consider any property $\mathrm{P}$ such that if $\rho\in \hom(\F_r, G)$ satisfies $\mathrm{P}$, then $g\rho g^{-1}$ and $\rho\cdot z$ satisfy $\mathrm{P}$, where $g\in G$ and $z\in Z(G)^r$. 
For $m\geqs 2$, or $m=0$, we have isomorphisms:
$$\pi_m(\XC{r}(G))\cong \pi_m(\XC{r}(DG))\cong \pi_m(\XC{r}(DK))\cong \pi_m(\XC{r}(K)).$$ 
For $m\geqs 2$, we have:
$$\pi_m(\XC{r}(G)^\mathrm{P}) \cong\pi_m(\XC{r}(DG)^\mathrm{P}) \textrm{ and } \pi_m(\XC{r}(K)^\mathrm{P})\cong \pi_m(\XC{r}(DK)^\mathrm{P}).$$ \end{thm}

\begin{proof}
By Theorem~\ref{deformretract}, the moduli spaces $\XC{r} (G)$ and $\XC{r} (K)$ have isomorphic $\pi_m$ for all $m\geqs  0$.  Now consider the  long exact homotopy sequence associated to the fibration in Theorem~\ref{P-theorem}, which has the form
$$\cdots\to\pi_m(\XC{r}(DG)^\mathrm{P})\to \pi_m(\XC{r}(G)^\mathrm{P})\to\pi_m(G/DG)\to\cdots\to\pi_0(G/DG) = 0.$$
Since $G/DG$ is a complex torus, $\XC{r} (G/DG)\heq (S^1)^n$ for some $n$, and hence $\pi_m (\XC{r} (G/DG)) = 0$ for $m\geqs  2$.  The long exact  sequence  now gives the desired isomorphism $\pi_m(\XC{r}(DG)^\mathrm{P})\stackrel{\isom}{\maps} \pi_m(\XC{r}(G)^\mathrm{P})$.  
Similarly, $\pi_m (\XC{r}(DK)^\mathrm{P})\stackrel{\isom}{\maps} \pi_m(\XC{r}(K)^\mathrm{P})$.  This completes the proof.
\end{proof}

We end this section by noting that our results regarding homotopy groups of the GIT quotient $\XC{r} (G)$ also apply to the ordinary quotient  $\hom(\F_r, G)/G$, despite the fact that this space need not be Hausdorff.

\begin{prop}\label{def-retr}
Let $G$ be a connected reductive $\C$--group, and let $X$ be a complex affine $G$--variety equipped with the Euclidean topology inherited from the embedding $X\subset \C^N$.  Then there is a strong deformation retraction from $X/G$ to $X^*/G \homeo X\aq G$, where $X^* \subset X$ is the subspace of points whose $G$--orbits are closed.
In particular, $X/G$ is homotopy equivalent to $X\aq G$.
\end{prop}
\begin{proof} We write $Gx$ to denote the $G$--orbit of $x\in X$, viewed as a subspace of $X$, and we write $[x]$ to denote the corresponding point in $X/G$.  By definition of $X^*$, $Gx$ is closed in $X$ if and only if $x\in X^*$, and furthermore these conditions hold if and only if $\{[x]\}$ is closed in $X/G$.
Recall that for each $G$--orbit $G x \subset X$, there exists a unique closed $G$--orbit $Gy \subset X^*$ such that $Gy \subset \overline{Gx}$, where $\overline{Gx}$ denotes the closure of this set in $X$, and the map $f\co X/G\to X/G$ given by $f([x]) = [y]$ is continuous.\footnote{This can be seen as follows.  Recall that $X\aq G$ is the quotient of $X/G$ by the relation $[x]\sim [y]$ if and only if $\overline{Gx} \cap \overline{Gy} \neq \emptyset$.  Let $\pi\co X/G\to X\aq G$ denote the  quotient map, and $i\co X^*/G \to X/G$ the inclusion.  As  discussed in the Introduction,
$\pi \circ i$ is a homeomorphism.  Now $f$ factors as $i\circ (\pi \circ i)^{-1} \circ \pi$.}

The deformation retraction  $H\co (X/G)\cross [0,1] \to X/G$ is defined by $H([x], t) = [x]$ for $t<1$, and $H([x], 1) = f([x])$.  Note that when $x\in X^*$, $f([x]) = [x]$, so we just need to check that $H$ is a continuous mapping.  Each closed set in $X/G$ has the form $C/G$ for some closed, $G$--invariant set $C \subset X$, and it suffices to show that $H^{-1} (C/G)$ is closed in $(X/G)\cross [0,1]$.  We have
$$H^{-1} (C/G) = \left((C/G)\cross [0,1)\right) \cup \left(f^{-1} (C/G) \cross \{1\}\right).$$
Say $x\in C$.  Since $C$ is $G$--invariant, we have $Gx \subset C$, and since $C$ is closed, we have 
$\overline{Gx} \subset C$.
Writing $f([x]) = [y]$, we have $y\in C$ as well, since $Gy \subset \overline{Gx} \subset C$.  Hence $H([x],1) = [y] \in C/G$ as well, so in fact we may write
$$H^{-1} (C/G) = \left( (C/G)\cross [0,1] \right) \cup \left(f^{-1} (C/G) \cross \{1\}\right).$$
Since $C/G$ is closed in $X/G$, the first set on the right is closed in $(X/G) \cross [0,1]$, and the second set is closed by continuity of $f$.\end{proof}

\section{The fundamental group}
\label{sah}

The fundamental group of
$\XC{r} (G)$ was first determined in \cite{BiLa}, and can also be obtained from more general considerations in \cite{BiLaRa}.

\begin{thm}[\cite{BiLa}]\label{simplyconnected}Let $G$ be a connected reductive $\C$--group, or a connected compact Lie group.  Then $\XC{r}(G)$ is path connected and 
$$\pi_1(\XC{r}(G))=\pi_1(G/DG)^r.$$
\end{thm}

In particular, $\XC{r}(\SLm{n})$ and $\XC{r}(\SUm{n})$ are simply connected, and $$\pi_1(\XC{r}(\GLm{n}))\cong \pi_1(\XC{r}(\Um{n}))\cong \mathbb{Z}^{r}.$$ 
Note that since $G/DG$ is a complex torus, $\pi_1 (G/DG)$ is always a free Abelian group of finite rank.

We proved in Lemma~\ref{pi_0} that the irreducible locus $\XC{r}(G)$, which coincides with the moduli space of stable representations
in the sense of affine GIT, is path connected.
In this section, we describe its fundamental group.

\begin{thm} \label{fund-thm}
Let $G$ be a connected reductive $\C$--group or a connected compact Lie group.  
Assume that either $r\geqs3$, or that $r\geqs2$ and $\mathrm{Rank}(DG)\geqs 2$.  Then 
\begin{equation}\label{pd}\pi_1 (\XC{r} (G)^{irr})\cong \pi_1 (G/DG)^r \times \pi_1 (\XC{r} (DG)^{irr}),\end{equation}
and $\pi_1 (\XC{r} (DG)^{irr})$ is a quotient of the finite Abelian group $\pi_1 (DG)^r$.  
In particular, if $\pi_1(DG)=1$, then $\pi_1 (\XC{r} (DG)^{irr})=1$, and
$$\pi_1 (\XC{r} (G)^{irr})\cong \pi_1(G)^r.$$
\end{thm}

We note that the product decomposition (\ref{pd}) is not canonical, as will be apparent from the proof.

\begin{proof}
By Theorem~\ref{gencodim}, the stated hypotheses imply that $\hom(\F_{r},G)^{red}$ has real codimension at least 4 inside of $\hom(\F_{r},G)$, so the inclusion $$\hom(\F_r,G)^{irr}\hookrightarrow \hom(\F_r,G)$$ is an isomorphism on fundamental groups (by transversality; see the discussion in Theorem \ref{locallysimplyconnected}).

We therefore have the following commutative diagram, in which the bottom row comes from the fibration sequence in Theorem \ref{P-theorem}:
\[
\xymatrix{\pi_{1}(DG)^r\ar@{=}[d]\ar@{^{(}->}[r] & \pi_{1}(G)^r\ar@{->>}[r]\ar@{=}[d] &\pi_1(G/DG)^r\ar@{=}[d]\\
\pi_1(\hom(\F_r,DG)^{irr})\ar@{^{(}->}[r] \ar@{->>}[d] & \pi_1(\hom(\F_r,G)^{irr})\ar@{->>}[r] \ar@{->>}[d]&\pi_1(\hom(\F_r,G/DG))\ar[d]^{\cong}\\
\pi_1(\XC{r}(DG)^{irr})\ar@{^{(}->}[r]& \pi_1(\XC{r}(G)^{irr})\ar@{->>}[r]&\pi_1(\XC{r}(G/DG)).
}
\] 

The vertical maps are surjective by a path-lifting argument similar to that in the proof of Theorem \ref{locallysimplyconnected}.  Since the fundamental group of a Lie group is Abelian, we conclude that the fundamental groups on the bottom row are also Abelian.  Since $\pi_1 (\XC{r} (G)) \isom \pi_1(G/DG)^r$ is free Abelian, the exact sequence on the bottom row splits (non-canonically) as 
$$\pi_1 (\XC{r} (G)^{irr})\cong\pi_1 (G/DG)^r\times \pi_1 (\XC{r} (DG)^{irr}),$$ 
as claimed.  Surjectivity implies that $\pi_1 (\XC{r} (DG)^{irr})$ is a quotient of $\pi_1(DG)^r$, which  is a finite Abelian group by Weyl's Theorem, and is trivial if $DG$ is simply connected.  Hence $\pi_1(\XC{r}(G)^{irr})\cong \pi_1(G/DG)^r \cong \pi_1(G)^r$ when $\pi_1 (DG) = 1$.  
\end{proof}

Recall that a subgroup $H < G$ is \e{irreducible} if it is not contained in any proper parabolic subgroup of $G$ (so a representation $\rho\co \F_r \to G$ is irreducible if and only if its image is an irreducible subgroup).  
Following Sikora in \cite{Si4}, we say $G$ has property {\it CI} if the centralizer of every irreducible subgroup of $G$ coincides with the center of $G$.  Clearly, $\SLm{n}$ and $\GLm{n}$ are CI, by Schur's lemma.  Sikora (\cite[Question 19]{Si4}) asks if there are any other reductive $\C$--groups with this property.

In \cite{Si4} and \cite{FlLa2} it is shown that orthogonal and symplectic groups are not CI, and in \cite{Si4} that finite quotients of groups that are not CI are not CI.  Moreover, it is clear that $G\times H$ is not CI if either $G$ or $H$ is not CI. 
We also extend the definition of CI group to the case of compact Lie groups.

\begin{thm} \label{fund-thm-CI}
Let $G$ be a connected reductive $\C$--group or a connected compact Lie group with property CI.  Assume that either $r\geqs3$, or that $r\geqs2$ and $\mathrm{Rank}(DG)\geqs 2$. 
Then 
$$\pi_1 (\XC{r} (G)^{irr})\cong \pi_1(G)^r,$$
and it follows that
$$\pi_1 (\XC{r} (DG)^{irr}) =\pi_1(DG)^r.$$
\end{thm}

For the proof, we need a lemma, provided to us by I. Biswas (personal communication).

\begin{lem}$\label{nullhomotopy}$ Let $G$ be a connected Lie group and let $PG = G/Z(G)$ act on $\hom(\F_r, G) = G^r$ by conjugation.  Then for each $\rho \in G^r$, the map $PG\to G^r$ given by $[g]\mapsto [g]\rho[g]^{-1}$ is nullhomotopic.
\end{lem}
\begin{proof} Since $G$ is connected, we may choose a path $\rho_t$ from $\rho\in G^r$ to the identity, and $[g]\mapsto [g]\rho_t [g]^{-1}$ is a nullhomotopy.
\end{proof}

\noindent \emph{Proof of Theorem \ref{fund-thm-CI}.}
If $G$ is CI, then $\XC{r} (G)^{irr} =  \XC{r} (G)^{good}$, and by Lemma \ref{lem22BL}
$$PG\to \hom(\F_r, G)^{irr}\to \XC{r}(G)^{irr}$$ 
is a $PG$--bundle (where the first map is the inclusion of an adjoint orbit).  Hence we have an exact sequence
\begin{equation} \label{es}\cdots \to \pi_1(PG)\to \pi_1(\hom(\F_r, G)^{irr})\to \pi_1(\XC{r}(G)^{irr})\to 0.\end{equation}
As discussed at the start of the proof of Theorem~\ref{fund-thm},
the map 
$$\pi_1(\hom(\F_r, G)^{irr}, \rho) \to \pi_1 (\hom(\F_r, G), \rho)$$
is an isomorphism, and the composite
$$\pi_1(PG)\maps \pi_1(\hom(\F_r, G)^{irr}) \maps \pi_1 (\hom(\F_r, G))$$
is zero by Lemma~\ref{nullhomotopy}.  Hence $\pi_1(PG)\to \pi_1(\hom(\F_r, G)^{irr})$ is also zero, and by exactness of (\ref{es}) we conclude
$$\pi_1(\XC{r}(G)^{irr}) \isom \pi_1(\hom(\F_r, G)^{irr})  \isom \pi_1 (G)^r.$$
The final statement then follows from (\ref{pd}).
The case of compact Lie groups is analogous.
$\hfill \Box$

\section{The linear case} \label{linear-sec}

In this section, $G_n$ will denote one of the groups $\GLm{n}, \SLm{n}, \Um{n}$, or $\SUm{n}$.
Here, we study the higher homotopy groups of $\XC{r} (G_n)^{irr}$ and $\XC{r} (G_n)$.

For these groups, a representation in $\hom(\F_r, G_n)$ is good if and only if it is irreducible.  In fact, by \cite{FlLa2}, for $(r-1)(n-1)\geqs 2$, we have $\XC{r} (\GLm{n})^{irr} = \XC{r} (\GLm{n})^{sm}$  and $\XC{r} (\SLm{n})^{irr} = \XC{r} (\SLm{n})^{sm}$.  For this reason, we will focus on the cases $(r-1)(n-1)\geqs 2$ in what follows.  In the compact cases,
we define $\XC{r}(\Um{n})^{sm} = \XC{r} (\GLm{n})^{sm} \cap \XC{r}(\Um{n})$, and similarly for $\SUm{n}$.  We note that this subspace coincides with the smooth locus of the character variety, viewed as a real analytic variety.  With the above restrictions on $r$ and $n$ we again have $\XC{r} (\Um{n})^{sm} = \XC{r} (\Um{n})^{irr}$ and $\XC{r} (\SUm{n})^{sm} = \XC{r} (\SUm{n})^{irr}$.

\begin{rem}\label{othercases} 
Suppose $(r-1)(n-1) < 2$. Then $\XC{r} (G_n)$ is a topological manifold $($with non-empty boundary in some cases$)$ and homotopy equivalent to either a product of some number of circles, or a point.  The proof of this is a good exercise for the reader.
\end{rem}

\subsection{Bott periodicity for the irreducible locus}

Since $\XC{r} (G_n)^{irr} = \XC{r} (G_n)^{good}$, Lemma \ref{lem22BL} shows that the sequence
\begin{equation}\label{PG-p}PG_n\maps \hom(\F_r, G_n)^{irr} \maps \XC{r} (G_n)^{irr}\end{equation}
is a principal $PG_n$--bundle.   We will use this fact to show that in a range of dimensions, the homotopy groups of $\X_r (G_n)^{irr}$ are $2$--periodic.

\begin{lem}\label{linear-codim} The real codimension of $\hom(\F_r, G_n)^{red}$ in $\hom(\F_r, G_n)$ is at least $2(r-1)(n-1)$.  Hence, by transversality, the inclusion map induces an isomorphism 
$$\pi_k  \left(\hom(\F_r, G_n)^{irr}\right) \stackrel{\isom}{\maps}  \pi_k  \left(\hom(\F_r, G_n)\right)\isom \pi_k (G_n)^r$$
for $k < 2(r-1)(n-1) -1$.
\end{lem}
\begin{proof} First we consider the case $G_n = \GLm{n}$.  Recall (Lemma~\ref{H_P}) that reducible locus is the union of finitely many irreducible algebraic sets of the form $H_P$, where $P < \GLm{n}$ is a maximal proper parabolic.  Hence it suffices to compute the codimensions of the $H_P$.
Each maximal proper parabolic $P < \GLm{n}$ is conjugate to one of the subgroups
$$P_k = \{A \in G_n \,:\, A(\C^k \cross \{0\}) \subset \C^k \cross \{0\}\},$$
where $k$ ranges from $1$ to $n-1$.
We have $\dim_\C (P_k) = k^2 + n(n-k)$, so the argument in the proof of Theorem~\ref{gencodim} shows that
$$\codim_{\C} (H_{P_k}) \geqs (r-1) (n^2 - k^2 - n(n-k)) = (r-1) (nk - k^2)$$
and this is minimized when $k = 1$ (or, symmetrically, $n-1$).  Hence
$$\codim_{\C} (\hom(\F_r, \GLm{n})^{red}) \geqs (r-1) (n - 1)$$
as claimed.

The arguments for $\SLm{n}$, $\Um{n}$, and $\SUm{n}$ are completely analogous.
\end{proof}

\begin{lem}\label{fiber-null}   For any $\rho\in \hom(\F_r, G_n)^{irr}$, the orbit-inclusion map $PG_n \to \hom(\F_r, G_n)^{irr}$, $[g] \mapsto [g]\rho [g]^{-1}$, induces the zero map on homotopy groups in dimensions less than $2(r-1)(n-1)-1$.
\end{lem}
\begin{proof} By Lemma~\ref{nullhomotopy}, the composite map
$$PG_n \to \hom(\F_r, G_n)^{irr} \to \hom(\F_r, G_n)$$
is nullhomotopic, and by Lemma \ref{linear-codim}, the second map in this composition is an isomorphism in the stated range.
\end{proof}

\begin{thm}\label{periodicity} Assume $(r-1)(n-1)\geqs 2$ and $1 < k < 2(r-1)(n-1)-1$.  Then
$$\pi_k (\XC{r}(G_n)^{irr}) = \begin{cases} \Z/n\Z, & \textrm{ if $k = 2$}\\
							\Z^r, & \textrm{ if $k$ is odd and $ k<2n$} \\
						   \Z, & \textrm{ if  $k$ is even and $2 < k<2n$}\\
						   (\Z/n!\Z)^r \oplus \Z, & \textrm{ if  $k=2n$} \,\,[\textrm{and }2n < 2(r-1)(n-1) - 1].\\
						   			\end{cases}
$$
Moreover, $\pi_k (\XC{r}(G_n)^{irr})$ is finite for $2n < k < 2(r-1)(n-1)-1$.

When $k$ is even and less than $2n$, the above isomorphisms are induced by the boundary map for the principal $PG_n$--bundle $(\ref{PG-p})$.
\end{thm} 

In particular, if $(r-1)(n-1)\geqs 2$ we have 
\begin{equation}\label{bdry} \xymatrix{ \pi_2 (\XC{r}(G_n)^{irr}) \ar[r]^-{\partial}_-{\isom} & \pi_{1}(PG_n) \isom \Z/n\Z},\end{equation}
 and for $r, n\geqs 3$ we have
$\pi_3 (\XC{r}(G_n)^{irr}) \isom \Z^r$. 
Furthermore, when $n=2$ we have $\pi_4 (\XC{r}(G_2)^{irr}) \isom (\Z/2\Z)^r \oplus \Z$ so long as $r\geqs 4$, and when $n\geqs 3$ we have $\pi_4 (\XC{r}(G_n)^{irr}) \isom \Z$ so long as $r\geqs 3$.

\begin{rem} The periodic groups $\Z$ and $\Z^r$ in Theorem~\ref{periodicity} are precisely the homotopy groups of the gauge group $\mathrm{Map} (\Sigma, G_n)$ of the trivial bundle $\Sigma \cross G_n$ $($with a shift in dimension$)$, where $\Sigma\heq \bigvee_r S^1$ is an open surface with fundamental group $\F_r$.
Analogous periodicity results $($again describing the homotopy of the moduli space of stable bundles in terms of a gauge group$)$ were established in~\cite{DU, BGG}. 
\end{rem}

\vspace{.2in}
\noindent \emph{Proof of Theorem \ref{periodicity}.}
By Lemmas \ref{linear-codim} and \ref{fiber-null}, if $k < 2(r-1)(n-1)-1,$ then the long exact sequence in homotopy associated to the bundle (\ref{PG-p}) breaks up into short exact sequences
\begin{equation}\label{per-seq}1\to  \pi_k (G_n)^r \to \pi_k \left(\XC{r} (G_n)^{irr}\right) \to \pi_{k-1} (PG_n)\to 1.\end{equation}
The short exact sequence of groups $1\to Z(G_n) \to G_n \to PG_n \to 1$ shows that
$\pi_1 (PG_n) \isom \Z/n\Z$, while $\pi_{k} (PG_n) \isom \pi_{k} (G_n)$ for $k\geqs 2$.
By Bott periodicity and stability of the homotopy groups of $G_n$, we have 
$$\pi_k (G_n) = \begin{cases} \Z, & \textrm{ if $k$ is odd and $1 < k<2n$} \\
						   0, & \textrm{ if  $k$ is even and $k<2n$,}\\
						   			\end{cases}$$
and by Borel--Hirzebruch \cite{BH} we have $\pi_{2n} (G_n) =    (\Z/n!\Z)$ $(n\geqs 2$).  
The calculation of $\pi_k (\XC{r} (G_n)^{irr})$ now follows easily from (\ref{per-seq}).

By a standard computation in rational homotopy theory (see \cite{Griffiths-Morgan}), $\pi_k (G_n) \otimes \mathbb{Q}$ vanishes for  $k > 2n-1$, so $\pi_k (G_n)$ is finite for $k > 2n-1$, as is $\pi_{k} (PG_n)$.  
 (Note that $\pi_k (G_n)$ is a finitely generated Abelian group, since it is isomorphic to $\pi_k (\Um{n})$ and $\Um{n}$ is a finite CW complex.)
Using the exact sequence (\ref{per-seq}), we see that 
$\pi_k (\XC{r} (G_n)^{irr})$ is finite in the stated range.
$\hfill \Box$

\begin{rem}
Work of Kervaire, Toda, and Matsunaga \cite{Kervaire, Toda, Matsunaga} gives calculations of $\pi_{2n+i} (G_n)$ for $i =1, \ldots, 5$, and Matsunaga \cite{Matsunaga-odd} has also provided a great deal of information on the odd primary parts of  $\pi_{2n+i} (G_n)$ for $i>5$.  In many of these cases, however, it is unclear whether the exact sequence $(\ref{per-seq})$ splits.  
\end{rem}

\begin{rem} It is natural to try to generalize the ideas in this section to any any connected reductive $\C$--group by replacing the irreducible representations by the good representations, since the good locus will provide a fibration $($but the irreducible locus in general will not$)$. Unfortunately, we do not know how to count the codimension of the complement of the good locus in general. However, we expect that $\pi_{2}(\X_{r}(G)^{good})\cong \pi_1(PG)$ for any connected reductive $\C$--group $G$ and for sufficiently large $r$. \end{rem}

\subsection{The second homotopy group of the full moduli space}

In his thesis \cite{Ba0}, Baird computed the Poincar\'{e} polynomial of $\XC{r} (\SUm{2})$.  This result is recalled in Section \ref{PP}, where we give an explicit version of his formula (Lemma \ref{Baird}).  One sees immediately from the formula that the Betti numbers satisfy
$$b_k (\XC{r} (\SUm{2})) = 0 \textrm{ for $1<k<6,$\, and \,} b_6 (\XC{r} (\SUm{2})) = \binom{r}{3}.$$  
Since $\XC{r} (\SUm{2})$ is simply connected (Theorem \ref{lem22BL}), Serre's Mod--$C$ Hurewicz Theorem implies that $\mathrm{Rank} (\pi_k (\XC{r} (\SUm{2}))) = b_k(\XC{r} (\SUm{2}))$ for $k \leqs 6$.  
Thus $\pi_6 (\XC{r} (\SUm{2}))$ is the first non-trivial homotopy group, and  in fact we expect that for all the Lie groups $G$ considered in this article, $\pi_k (\XC{r} (G)) = 0$ for $2 \leqs k \leqs 5$.

In this section, we will provide \e{integral} information about the second homotopy group (and consequently the second homology group) of $\XC{r} (\SUm{n})$.
Note that by Theorem~\ref{higherhomotopy}, replacing $\SUm{n}$ by $\Um{n}$, $\GLm{n}$, or $\SLm{n}$ does not change 
$\pi_2$.
 To obtain information about this group, we will use a general-position type argument in the context of simplicial complexes.  We begin by explaining the necessary ingredients.

Consider a simplicial complex $X$ together with a subcomplex $Y\subseteq X$ such that $X\sm Y$ is dense.  Under certain local conditions around points in $Y$, it is possible to homotope every map $S\to X$ off of $Y$, where $S$ is a compact manifold of dimension at most 2.   This result is inspired by Lemma 2.5 from Gomez--Pettet--Souto~\cite{GPS}, which deals with maps $S^1\to X$.
\begin{prop}\label{trans-prop}

Let $X$ be a simplicial complex and let $Y\subseteq X$ be a subcomplex.  Assume the following conditions hold:
\begin{enumerate}
\item[$(1)$] The space $X \sm Y$  is dense in $X$.
\item[$(2)$]  For each point $y\in Y$, and each open neighborhood $U \subseteq X$ containing $y$, there exists an open neighborhood $V\subseteq U$ $($with $y\in V)$ such that $V\setminus Y$ is path connected.
\item[$(3)$] For each point $y\in Y$, there exists a path connected open neighborhood $V$ of $y$ such that
 $\pi_2 (V) = 0$ and the natural map $\pi_1 (V\setminus Y, v) \to \pi_1 (V,v)$ is injective $($for each $v\in V\sm Y)$. 
\end{enumerate}
Let $S$ be a compact manifold of dimension at most $2$ with $($possibly empty$)$ boundary and let $S' \subset S$ be a closed subset containing $\partial S$ such that $S$ admits a triangulation with $S'$ as a subcomplex.
Then for every continuous map $f\co (S, S') \to (X, X\sm Y)$ such that $f$ is simplicial on $S'$, there exists a map $h\co S \to X$ such that $h$ is homotopic to $f$ $($rel $S'$$)$ and $h (S) \cap Y = \emptyset$.
\end{prop}

The proof of Proposition \ref{trans-prop} will be given in the Appendix.

\begin{thm}\label{trans}
Let $Y\subset X$ be as in Proposition \ref{trans-prop}.  
Then the inclusion $X\sm Y\injects X$ is $2$--connected.
\end{thm}
\begin{proof}  
The fact that the maps are surjective on $\pi_i$ for $i\leqs  2$ follows from Proposition~\ref{trans-prop} by taking $(S, S')$ to be the based sphere $(S^i, *)$.  Note that given a map
$f\co (S^i, *) \to (X, X\sm Y)$, we can always subdivide $X$ to make $f(*)$ a vertex (so that $f$ is simplicial on $S' = *$).

To see that these maps are injective on $\pi_i$ for $i\leqs  1$, we will apply the Proposition in the case $(S, S') = (S^i \cross [0,1], S^i \cross \{0, 1\})$.  Assume $i=1$; the case $i=0$ is simpler.  Given $f, g\co S^1 \to X\sm Y$ and a homotopy between them in $X$, we need to show that there exists a homotopy between them in $X\sm Y$.  Since $f(S^1)$ and $g(S^1)$ are compact, they are contained inside some finite subcomplex $X' \subset X$.  These compact sets are disjoint from the compact set $Y\cap X'$, so they lie 
at some positive distance from $Y\cap X'$ (in the simplicial metric on $X'$) and hence after barycentrically subdividing $X$ enough times (so that the diameter of each simplex is sufficiently small), every simplex in $X'$ meeting $f(S^1)$ or $g(S^1)$  is disjoint from $Y$.  Viewing $f$ and $g$ as maps from $S^1$ to the subcomplex $X'' \subset X\sm Y$ consisting of all maximal simplices in $X'$ that meet $f(S^1)$ or $g(S^1)$, the Simplicial Approximation Theorem allows us to choose a triangulation of $S^1$ and maps $f', g'\co S^1 \to X''$, simplicial with respect to this triangulation, 
such that there exist homotopies $f\heq f'$ and $g\heq g'$ inside $X''$ (hence inside $X\sm Y$).  Now, since $f$ and $g$ are homotopic in $X$, so are $f'$ and $g'$.   Let $H$ be a homotopy between $f'$ and $g'$.  We can now choose a triangulation of $S^1 \cross [0,1]$ that restricts, on $S^1\cross \{0\}$ and on  $S^1\cross \{1\}$, to the triangulation of $S^1$ for which $f'$ and $g'$ are simplicial.  Applying Proposition~\ref{trans} to $H\co (S^1 \cross [0,1], S^1 \cross \{0, 1\})\to (X, X\sm Y)$, we conclude that $f'$ and $g'$ are homotopic inside $X\sm Y$, and it follows that the same is true of $f$ and $g$.
\end{proof}

We note that when $S$ is $1$--dimensional, only the first two conditions of Proposition~\ref{trans-prop} will be used in its proof, so Proposition~\ref{trans-prop} extends~\cite[Lemma 2.5]{GPS}.  Moreover, when $S$ is 
$0$--dimensional only the first condition is used. This motivates the following question.

\begin{que}
Can Proposition~\ref{trans} be extended to maps $M \to X$, with $M$ a closed, triangulable, manifold of dimension $n> 2$, after adding the condition that each point in $Y$ admits a neighborhood $V$ such that $\pi_i (V) = 0$ for $i=1,\ldots, n$ and $\pi_{i} (V\sm Y) \injects \pi_{i} (V)$ for $i < n?$
\end{que}

We now apply the general result  above to the moduli spaces considered in this paper.

\begin{prop}\label{2-connect} Assume that $(r-1)(n-1)\geqs 2$. Then the inclusion $\XC{r} (G_n)^{irr} = \XC{r} (G_n)^{sm} \injects \XC{r} (G_n)$ is a $2$--connected map.
\end{prop}

Recall that in these cases, the good locus in $\XC{r} (G_n)$ coincides with the irreducible locus.

\begin{proof} 
By assumption $n\geqs 2$; however, we note that when $n=1$, $G_1$ is Abelian and $PG_1=1$, so $\XC{r} (G_1)^{irr} = \XC{r} (G_1)$ and there is nothing to prove.

For $G_n = \Um{n}$, the stabilizer of every representation is connected (see \cite[Lemma 4.3]{Ramras2}).  Therefore, Lemma \ref{locallysimplyconnected} implies that   conditions (2) and (3) in Proposition \ref{trans-prop} are satisfied.  
Condition (1) is satisfied as well, as explained in more generality at the end of the proof of Theorem~\ref{disconnected}.  Therefore, Theorem \ref{trans} gives the result.

Next consider the case $G_n = \GLm{n}$.  In the diagram
$$\xymatrix { \pi_1(\XC{r} (\Um{n})^{irr}) \ar[r]^\isom \ar[d]&  \pi_1 (\XC{r} (\Um{n}))\ar[d]\\
		\pi_1 (\XC{r} (\GLm{n})^{irr}) \ar[r] & \pi_1 (\XC{r} (\GLm{n})),}$$
the right-hand vertical map is an isomorphism because $\XC{r} (\GLm{n})$ deformation retracts to $\XC{r} (\Um{n})$ (Theorem \ref{deformretract}).  Hence the map $\pi_1 (\XC{r} (\GLm{n})^{irr}) \to \pi_1 (\XC{r} (\GLm{n}))$ is a surjection, but we know that these groups are free Abelian of the same rank by Theorems \ref{simplyconnected} and \ref{fund-thm}.  Hence the map must be an isomorphism.  Surjectivity of the map $\pi_2 (\XC{r} (\GLm{n})^{irr}) \to \pi_2 (\XC{r} (\GLm{n}))$ is proven similarly.

Finally, suppose $G_n = \SUm{n}$.  By Theorems \ref{simplyconnected} and \ref{fund-thm} we know that $\XC{r} (\SUm{n})$ and $\XC{r} (\SUm{n})^{irr}$ are simply connected.  The desired statement for $\pi_2$ follows from the diagram
$$\xymatrix { \pi_2(\XC{r} ( \SUm{n})^{irr}) \ar[r] \ar[d]&  \pi_2 (\XC{r} (\SUm{n})) \ar[d]\\
		\pi_2 (\XC{r} (\Um{n}))^{irr} \ar@{->>}[r] & \pi_2 (\XC{r} (\Um{n})),}$$
in which the vertical maps are isomorphisms by Theorem \ref{higherhomotopy}.
Finally, the case of $\SLm{n}$ can be deduced from the $\GLm{n}$ case analogously.
\end{proof}

\begin{thm}\label{pi2} Let $G_n$ be one of $\GLm{n},\SLm{n},\SUm{n}$
or $\Um{n}$. Then we have $\pi_{2}(\XC{r}(G_n))= 0$.
\end{thm}

\begin{proof} By Theorem \ref{higherhomotopy}, it suffices to prove
the case of $G_n = \SUm{n}$. When $(r-1)(n-1) < 2$ the result is trivial since $\X_r(\SUm{n})$ is contractible (more generally see Remark \ref{othercases}).  So we assume $(r-1)(n-1)\geqs 2$.

For $r'< r$, the projection $q\co \F_r \to \F_{r'}$ defined by $e_i \mapsto e_i$ for $i\leqs r'$, $e_i \mapsto 1$ for $i> r'$ induces
a map $q^\#\co \X_{r'}(\SUm{n})\to\X_{r}(\SUm{n})$, and since smoothness coincides with irreducibility in the situations under consideration, and applying $q^\#$ does not change the image of a representation,  
we see that $q^\#$ restricts to a map $\XC{r'} (\SUm{n})^{sm} \to \XC{r} (\SUm{n})^{sm}$.  We now have a commutative diagram
\begin{equation} \label{r-r'}
\xymatrix{\pi_{2}(\X_{r'}(\SUm{n})^{sm})\ar[r]\ar@{->>}[d] & \pi_{2}(\X_{r}(\SUm{n})^{sm})\ar@{->>}[d]\\
\pi_{2}(\X_{r'}(\SUm{n}))\ar[r]^-{q^\#_*} & \pi_{2}(\X_{r}(\SUm{n})),
}
\end{equation}
in which the vertical maps are surjective by Proposition~\ref{2-connect}, and the top map is an isomorphism between finite cyclic groups by naturality of \eqref{bdry}.
This implies that $q^\#_*$ is surjective, and $q^\#_*$ is also injective since $q$ is split by the map $F_{r'} \to F_r$ defined by $e_i \mapsto e_i$.  Hence $q^\#_*$ is an isomorphism, and we conclude that the orders of the finite cyclic groups 
$\pi_{2}(\X_{r}(\SUm{n}))$ are  independent of $r$ (so long as $(r-1)(n-1)\geqs 2$).  

When $n=2$,  $\X_3(\SUm{2})$ is homotopy equivalent to a 6--sphere (see \cite{FlLa}), so $\pi_2 (\X_3 (\SUm{2})) = 0$ and hence 
$\pi_2 (\X_r (\SUm{2})) = 0$ for all $r\geqs 3$.

Now assume $n\geqs 3$.  Since $\pi_{2}(\X_{r}(\SUm{n}))$ and $\pi_{2}(\X_{r'}(\SUm{n}))$ are finite groups of the same order, every homomorphism between them that is split (admits a left or right inverse) is an isomorphism.  
Consider the $\SUm{n}$--equivariant maps
$$\hom(\F_2, \SUm{n}) = \SUm{n}^2 \stackrel{\mu}{\maps} \hom(\F_3, \SUm{n}) = \SUm{n}^3 \stackrel{\eta}{\maps} \hom(\F_2, \SUm{n}) = \SUm{n}^2$$
defined by
$$\mu (A, B) = (A, B, AB), \,\,\,\, \eta (A, B, C) = (AB, C).$$
These maps are split (respectively) by the $\SUm{n}$--equivariant maps 
$$(A, B, C) \mapsto (A,B) \,\, \textrm{ and } (A, B) \mapsto (A, I, B),$$
 so the induced maps $\overline{\mu} \co \XC{2} (\SUm{n}) \to  \XC{3} (\SUm{n})$ and $\overline{\eta} \co \XC{3} (\SUm{n}) \to  \XC{2} (\SUm{n})$ are also split, and hence $ \overline{\eta} \circ \overline{\mu}$ induces an isomorphism on $\pi_2 ( \XC{2} (\SUm{n}))$.  Since
$\overline{\eta} \circ \overline{\mu} ([A, B]) = [AB, AB]$, 
  we have a factorization
$$\xymatrix{ \XC{2} (\SUm{n}) \ar[rr]^-{\overline{\eta} \circ \overline{\mu}} \ar[dr]^\alpha  && \XC{2} (\SUm{n}) \\ & 
\XC{1} (\SUm{n}) \ar[ur]^\beta},$$
where $\alpha ([A, B]) = [AB]$ and $\beta ([A]) = [A, A]$.  However, $\XC{1} (\SUm{n})$ is contractible, since it is homeomorphic to the Weyl alcove (see \cite[p. 168]{DK}), and in particular $\pi_2 (\XC{1} (\SUm{n}) ) = 0$.  This completes the proof.
\end{proof}

\begin{rem}  A small modification to the above argument can be used to deduce the $n=2$ case directly.  We also note that the $n=3$ case can be handled by appealing to  the fact that  $\X_2(\SUm{3})$ is homotopy equivalent to an $8$--sphere \cite{FlLa}.
\end{rem}

Applying the Hurewicz Theorem gives the following corollaries for the special linear case.

\begin{cor}\label{betti-2} Let $G_n$ be either $\SLm{n}$ or $\SUm{n}$. Then
$$H_2 (\XC{r}(G_n); \Z)  = H^2(\XC{r}(G_n);\Z) = 0.$$
\end{cor}
\begin{proof}
Since $\XC{r}(\SLm{n})$ and $\XC{r}(\SUm{n})$ are simply connected, Hurewicz Theorem 
provides isomorphisms $\pi_{2}(\X_{r}(G_n))\cong H_{2}(\X_{r}(G_n);\Z)$, 
and the Universal Coefficient Theorem further implies that $H^2(\XC{r}(G_n);\Z)=0$.
\end{proof}

Recent results in \cite{CFLO} allow us to extend our analysis of $\pi_2$ to certain real reductive Lie groups.

\begin{cor}
Let $G = \mathsf{U}(p,q)$ or $\mathsf{Sp}(2n,\R)$.  Then $\pi_2(\XC{r}(G)) = 0$.
\end{cor}

\begin{proof}
Since the maximal compact subgroup of $\mathsf{U}(p,q)$ is $\mathsf{U}_p\times \mathsf{U}_q$, the main result in \cite{CFLO} shows that 
$\XC{r}(\mathsf{U}(p,q))$ is homotopic to $\XC{r}(\mathsf{U}_p)\times \XC{r}(\mathsf{U}_q)$.  Therefore, Theorem \ref{pi2} shows that $\pi_2(\XC{r}(\mathsf{U}(p,q))) = 0$.

The result for $\XC{r}(\mathsf{Sp}(2n,\R))$ follows similarly, since this space is homotopy equivalent to $\XC{r}(\mathsf{U}_n)$.
\end{proof}

Based on the above, we make the following conjecture.  Note that complex reductive $\C$--groups and compact Lie groups are real reductive.

\begin{conj}\label{conj}
Let $G$ be a connected real reductive Lie group.  Then $\pi_2(\X_r(G))$ is trivial.
\end{conj}

\begin{rem}$\label{Ste-rmk}$
For $r=1$ the above conjecture is true since $G\aq G$ is contractible if $G$ is simply connected: this follows from results in \cite{Ste}. 
\end{rem}

We end this section by noting that the codimension of the smooth locus in $\X_r (G_n)$ grows linearly in both $r$ and $n$.  Theorems \ref{periodicity} and \ref{pi2} show that the map $\X_r (G_n)^{sm} \to \X_r (G_n)$ is not 3--connected when $n\geqs 2$, so this gives examples in which the inclusion of the smooth locus is not highly connected, despite the singular locus lying in arbitrarily high codimension.

\begin{prop}\label{codim} Assume $r, n\geqs 2$.  Let $G_n$ be either
$\GLm{n}$ or $\SLm{n}$. Then the  complex 
codimension of $\XC{r}(G_n)^{sing}$ is at least $3$, and grows linearly
in $r$ and $n$. Similarly, for $G_n$ equal to $\Um{n}$ or $\SUm{n}$
the real codimension of $\XC{r}(G_n)^{sing}$ is at least $3$ and grows linearly.\end{prop}

\begin{proof} First consider the case $G_n = \GLm{n}$.
If $r\geqs2$ and $n\geqs3$, or $r\geqs3$ and $n\geqs2$, then $\XC{r}(\GLm{n})^{sing} = \XC{r}(\GLm{n})^{red}$
and an argument analogous to the proof of Lemma~\ref{linear-codim} gives
\begin{eqnarray*}
\mathrm{codim}_{\C} \left(\XC{r}(\GLm{n})^{sing}\right) & = & \dim_{\C} \left(\XC{r}(\GLm{n})\right) - \dim_{\C} \left(\XC{r}(\GLm{n})^{sing}\right)\\
 & = & (n^{2}(r-1)+1)-((n-1)^{2}(r-1)+1+r)\\
 & = & 2(n-1)(r-1)-1\\
 & \geqs & 3.
\end{eqnarray*}
In the case that $r=2=n$, the functions $\tr{A}, \tr{B}, \tr{AB}, \det(A)$, and $\det(B)$ induce an isomorphism $\mathfrak{gl}_2\aq\GLm{2} \isom \C^5$~\cite{DrFo}, and hence
$\XC{2}(\GLm{2})\cong\C^{3}\times(\C^{*})^{2}$
is smooth and so the complex codimension of the singular locus
is 5. Thus, $\mathrm{codim}_{\R} \left(\XC{2}(\GLm{2})^{sing}\right)\geqs6$. 

Restricting the determinant to be 1, we obtain for $r\geqs2$ and $n\geqs3$,
or $r\geqs3$ and $n\geqs2$: 
\begin{eqnarray*}
\mathrm{codim}_{\C}\left(\XC{r}(\SLm{n})^{sing}\right) & = & \dim_{\C}\left(\XC{r}(\SLm{n})\right)-\dim_{\C}\left(\XC{r}(\SLm{n})^{sing}\right)\\
 & = & (n^{2}-1)(r-1)-((n-1)^{2}(r-1)+1)\\
 & = & 2(n-1)(r-1)-1\\
 & \geqs & 3.
\end{eqnarray*}
If $r=n=2$, the above description of $\XC{2}(\GLm{2})$ shows that $\XC{2}(\SLm{2})\cong\C^{3}$ and
is smooth, so the complex codimension of $\XC{2}(\SLm{2})^{sing}$ is 3.

The cases of $\Um{n}$ and $\SUm{n}$ are similar.
\end{proof}

\section{An application to Lie groups}\label{app-sec}

In this section, we apply our topological methods to study centralizers of subgroups in compact connected Lie groups and connected reductive $\C$--groups.  As noted previously, every subgroup $H\leqs \Um{n}$ has connected centralizer \cite[Lemma 4.3]{Ramras2}.  As observed in \cite[Section 1.2]{Humphreys-conjugacy}, the same is true for $\GLm{n}$ if one restricts attention to finitely generated subgroups $H \leqs \GLm{n}$, because the centralizer $C_{\GLm{n}} (H)$ is Zariski open in the vector space
$$\{A\in M_n (\C) \,:\, AhA^{-1} = h \textrm{ for all } h\in H\}.$$

To our knowledge, $\Um{n}$ and $\GLm{n}$ are the only known groups with such properties.  It is a deep theorem of Springer and Steinberg that semisimple elements in simply connected semisimple algebraic groups over $\C$ have connected centralizers (for an exposition, see \cite{Humphreys-conjugacy}).  Here we will show that the presence of torsion in $\pi_1 (G)$ forces the existence of subgroups whose centralizers are \e{disconnected} 
(even after killing the center of $G$).

\begin{thm}\label{disconnected} Let $G$ be either a connected reductive $\C$--group, or a compact connected Lie group.  If $\pi_1 (DG) \neq 1$, then there exists a finitely generated subgroup $H \leqs G$ such that $C_G (H)/Z(G)$ is disconnected.
\end{thm}

Note that for every finitely generated subgroup $H \leqs G$, we have $C_G (H) = C_G (\overline{H})$, where $\overline{H}$ is the Zariski closure of $H$ in $G$.  Hence in fact Theorem~\ref{disconnected} implies that there exist Zariski closed subgroups in $G$ with disconnected centralizers.

\begin{proof}  Let $K$ be a maximal compact subgroup of $G$ (so $K = G$ if $G$ is compact). 
Assume, for a contradiction, that for every finitely generated subgroup $H\leqs G$, the quotient $C_G (H)/Z(G)$ is connected.  
Lemma~\ref{centralizer-SDR} implies that  $C_K (H)/Z(K)$ is connected for every finitely generated subgroup $H\leqs K$.  
Hence $K$ is a CI group, and by Theorem \ref{fund-thm-CI}, we have
$$\pi_1 (\XC{r} (K)^{irr}) \isom \pi_1 (K)^r$$
for $r\geqs 3$.

On the other hand, we claim that the hypotheses of  Theorem~\ref{trans} are satisfied, meaning that the map
$$\pi_1 (\XC{r} (K)^{irr}) \maps \pi_1 (\XC{r} (K)) \isom \pi_1 (K/DK)^r$$
is an isomorphism.  This is impossible, since the fibration sequence $$DK \to K \to K/DK$$ yields a split short exact sequence on fundamental groups, and $\pi_1 (DK) \isom \pi_1 (DG) \neq 0$ by hypothesis.

To check that the hypotheses of Theorem~\ref{trans} are satisfied, first note that our hypotheses on $G$ guarantee that for sufficiently large $r$, we have $\codim_{\mathbb{R}} \left(\hom(\F_r, K)^{irr}\right) \geqs 4$ (see Theorem~\ref{gencodim}).
Hence $\hom(\F_r, K)^{irr}$ is dense in the smooth variety $\hom(\F_r, K) \isom K^r$, and hence $\XC{r} (K)^{irr}$ is dense in $\XC{r} (K)$.
Next, since the image of a reducible representation $\rho\in \hom(\F_r, K)$ is a finitely generated subgroup, we know that 
$\mathrm{Stab}(\rho)/Z(K) = C_K (\textrm{Im} (\rho))/Z(K)$ is connected, and by Lemma \ref{locallysimplyconnected} there exist arbitrarily small contractible neighborhoods of $[\rho]$ in $\XC{r}(K)$ whose irreducible points form a simply connected subset.
\end{proof}

\section{The singular locus}\label{sing-sec}

In \cite{FlLa2}, it is shown for $G=\SLm{n}$ or $\GLm{n}$ that $$\XC{r}(G)^{sm}=\XC{r}(G)^{good}=\XC{r}(G)^{irr}$$ if $(r-1)(n-1)\geqs 2$. As shown in \cite{HP}, this result does not even generalize to $G=\p\SLm{2}$.
In this section, we address the following conjecture.

\begin{conj}[Conjectures 3.34 and 4.8 in \cite{FlLa2}] \label{red-conj}  If $r\geqs 3$, or $r\geqs 2$ and $\mathrm{Rank}(G)$ is sufficiently large, then
 $$\XC{r}(G)^{red}\subset \XC{r}(G)^{sing}.$$
 \end{conj}
Note that \cite[Remark 3.33]{FlLa2} shows that $\XC{2}(\PSLm{2})$ has smooth points which are reducible and singular points which are irreducible; so a condition on the rank of $G$ when $r=2$ is necessary.  However, the $r=2$ case of this conjecture requires modification, due to the examples below.  We prove a modified version in Theorem \ref{conj-thm}.

\begin{exam}
Let $G$ be any connected reductive $\C$-group. Then there exists a reducible smooth point and a irreducible singular point in $\XC{2}(G\times \PSLm{2})$.  Let $[\rho]\in \XC{2}(G)^{good}$ and let $[\psi_1]\in \XC{2}(\PSLm{2})^{red}\cap \XC{2}(\PSLm{2})^{sm}$ and $[\psi_2]\in \XC{2}(\PSLm{2})^{irr}\cap \XC{2}(\PSLm{2})^{sing}$.  Then clearly, $[\rho\oplus \psi_1]$ has positive dimensional stabilizer and so is reducible, but yet it is in $\XC{2}(G)^{sm}\times \XC{2}(\PSLm{2})^{sm}$ and so is a smooth point.  On the other hand, $[\rho\oplus \psi_2]$ has a finite stabilizer and so is irreducible $($but not good$)$, but is in $\XC{2}(G)^{sm}\times \XC{2}(\PSLm{2})^{sing}$ and so is a singular point.

This shows that there are Lie groups $H$ of arbitrarily large rank with the property that $\XC{2}(H)$ has smooth reducibles and singular irreducibles.
\end{exam} 
 
We now give an example to show that the groups $H$ in the above example do not only arise as products with rank 1 Lie groups.

\begin{exam}
In \cite{Si6}, $\XC{2}(\mathsf{SO}_4(\C))=\hom(\F_2,\mathsf{SO}_4(\C))\aq \mathsf{SO}_4(\C)$ is explicitly described.  Consider 
$$\rho=\left(\frac1{12}\left(
\begin{array}{cccc}
  {37}  &  {35 i}  & 0 & 0 \\
 - {35 i}  &  {37}  & 0 & 0 \\
 0 & 0 &  {13}  &  {5 i}  \\
 0 & 0 & - {5 i}  &  {13}  \\
\end{array}
\right), \, \, \,  \frac1{40}\left(
\begin{array}{cccc}
  {401}  &  {399 i}  & 0 & 0 \\
 - {399 i}  &  {401}  & 0 & 0 \\
 0 & 0 &  {41}  &  {9 i}  \\
 0 & 0 & - {9 i}  &  {41}  \\
\end{array}
\right)\right)$$ 
in $\hom(\F_2,\mathsf{SO}_4(\C))$.  It is clearly reducible and completely reducible.  
However, as the relations and generators are explicitly computed in \cite{Si6} for this variety, in {\it Mathematica} we can compute the rank of the Jacobian matrix at $\rho$, finding it to be $11$.  There are $17$ generators in the presentation for the coordinate ring and the variety has dimension $6$.  Thus, $\rho$ is a smooth point if and only if the rank is $11$.  Therefore, we have a smooth point in $\X_2 (\mathsf{SO}_4(\C))$ that is not in $\X_2 (\mathsf{SO}_4(\C))^{irr}$,  arising from a reductive group of semisimple rank $2$ that is not a product with a rank $1$ group.  However, note that the simple factors are each of rank $1$.

\end{exam}

The examples above show that for $r=2$ the rank of the Lie group $G$ being large is not sufficient for Conjecture \ref{red-conj} to hold, although we do believe the conjecture for $r\geqs3$ without any condition on the rank of $G$.

We now mostly resolve Conjecture \ref{red-conj} by showing that if each simple factor in the derived subgroup of $G$ has rank at least 2 then the reducibles are always singular.  For $r=2$, we expect this to be a sharp condition on the Lie group $G$.

\begin{thm}\label{conj-thm} Let $r\geqs 2$.
If $G$ is a connected reductive $\C$--group such that the Lie algebra of $DG$ has simple factors of rank 2 or more, then:
\begin{enumerate}
 \item[$(1)$] $\XC{r}(G)^{red}\subset \XC{r}(G)^{sing}$,
 \item[$(2)$] $\XC{r}(G)^{irr}-\XC{r}(G)^{good}\subset \XC{r}(G)^{sing}$, and all points in $\XC{r}(G)^{irr}-\XC{r}(G)^{good}$ are orbifold singularities,
 \item[$(3)$] $\XC{r}(G)^{good}=\XC{r}(G)^{sm}$.
\end{enumerate}

\end{thm}

\begin{proof}
By \cite[Theorem 8.9]{Ri} when $G$ is semisimple and the simple summands of the Lie algebra $\mathfrak{g}$ have rank at least 2, and $r\geqs 2$, we have $\XC{r}(G)^{good}=\XC{r}(G)^{sm}$.  Therefore, Proposition \ref{etale} and Corollary \ref{P-cor} together imply that $\XC{r}(G)^{good}=\XC{r}(G)^{sm}$ for any connected reductive $\C$--group whose derived group has simple factors of rank at least 2.

We then conclude that $\XC{r}(G)^{red}\subset \XC{r}(G)^{sing}$ since the $PG$--stabilizer of any reducible representation has positive dimension (hence they are not good).  For the same reason, $\XC{r}(G)^{irr}-\XC{r}(G)^{good}\subset \XC{r}(G)^{sing}$, and by Lemma \ref{lem22BL} these singularities are of orbifold type.  
\end{proof}

We have the following geometric corollary.

\begin{cor}
Let $DG$ have simple factors of rank at least $2$, and assume $r\geqs 2$.
If $G$ is not CI, then the orbifold singular locus in $\XC{r}(G)^{irr}$ is non-empty.
\end{cor}

It remains to deal with groups $G$ whose derived subgroup has rank 1 factors.  Given the classification of singularities in $\XC{r}(\SLm{2})$ by \cite{FlLa2}, and in $\XC{r}(\PSLm{2})$ by \cite{HP}, we expect the above theorem to extend directly to $r\geqs3$ and $DG$ having rank 1 simple factors.

 As a step in this direction,  we reduce Conjecture \ref{red-conj} to the case when $G$ is semisimple and simply connected, leading to a proof of the conjecture for $G = \PSLm{2}$.  We need two lemmas.

\begin{lem}\label{finitequotientsing}
Let $G$ be a connected, reductive $\C$--group and let $Z < Z(G)$  be  a finite central subgroup.  Denote the quotient map $G\to G/Z$ by $\pi$, and consider the induced map $\pi_* \co \XC{r}(G) \to \XC{r} (G/Z)$.  If $\XC{r}(G)^{red}\subset \XC{r}(G)^{sing}$, then $\pi_* (\XC{r} (G)^{red}) \subset \XC{r} (G/Z)^{sing}$.
\end{lem}

\begin{proof}
By Lemma~\ref{H_P}, the sets $H_P:=\cup_{g\in G}\hom(\F_r,gPg^{-1})$, with $P\leqs G$ a maximal proper parabolic subgroup, are exactly the irreducible components of $\XC{r}(G)^{red}$.  
Note that since the identity is in $P$ for every parabolic subgroup $P < G$, the trivial representation is in each irreducible component $H_P\aq G$.  Moreover, the left multiplication action of $Z^r$ on $\XC{r}(G)$ is free at the trivial representation.  Note that the set of representations where $Z^r$ does not act freely is the union of the Zariski closed sets determined by the algebraic equations $z[\rho]=[\rho]$, and thus is a Zariski closed set (recall that $Z$ is finite).  Since in every component there is a point where the action is free (the trivial representation) the action of $Z^r$ on $\XC{r}(G)^{red}$ is free on a non-empty (dense) Zariski open set in each component.  Therefore, the quotient mapping $\XC{r}(G)\to \XC{r}(G)/Z^r\cong \XC{r}(G/Z)$ is generically \'etale on the restriction to $\XC{r}(G)^{red}$; that is, there is an open dense set $U\subset \XC{r}(G)^{red}$ for which the mapping is \'etale.  Therefore, every $[\rho]\in U$ (necessarily singular by assumption) is mapped to a singular point.  Since the quotient map is continuous, the closure of $U$ maps into the closure of the image of $U$.  However, the closure of $U$ is $\XC{r}(G)^{red}$ and the closure of the image of $U$ is contained in the singular locus (since the singular locus is a closed set).
\end{proof}

\begin{lem} \label{red-lem} Let $G$ be a connected, reductive $\C$--group and let $Z < Z(G)$  be  a Zariski closed, central subgroup.  Denote the quotient map $G\to G/Z$ by $\pi$, and consider the induced map $\pi_* \co \XC{r}(G) \to \XC{r} (G/Z)$.  Then
$\pi_* (\XC{r} (G)^{red}) = \XC{r} (G/Z)^{red}$.
\end{lem}

\begin{proof}  It will suffice to show that a representation $\rho\in\hom(\F_r, G)$ is reducible if and only if $\pi\circ \rho \in \hom(\F_r, G/Z)$ is reducible, or equivalently that a subgroup $P<G$, with $Z < P$, is parabolic if and only if $\pi(P) = P/Z$ is parabolic in $G/Z$.  Recall that a subgroup $Q<H$ is parabolic ($H$ a reductive $\C$--group) if and only if $Q$ is Zariski closed in $H$ and $H/Q$ is projective.  

First, note that a subgroup $Q<G$, with $Z<Q$, is Zariski closed if and only if $\pi(Q)=Q/Z<G/Z$ is Zariski closed.  This follows since $\pi$ is the GIT quotient map, and $Q$ is a $Z$--space (see \cite[Proposition 6.2]{Do}).

Now, say $Z< P < G$ and $P$ is Zariski closed.  We claim that there is an isomorphism of algebraic varieties $G/P \isom (G/Z)/(P/Z)$.  
We will appeal to the universal property of the quotient map $p\co H\to H/Q$, where $H$ is an algebraic $\C$--group and $Q<H$ is a Zariski closed subgroup: as shown in \cite{Hum}, this map has the property that if $f\co H\to X$ is an algebraic map whose fibers $f^{-1} (x)$ are all unions of cosets of $Q$, then there exists a unique algebraic map $\overline{f}\co H/Q\to X$ such that $f = \overline{f} \circ p$.  
Let $q\co G/Z \to (G/Z)/(P/Z)$ be the quotient map (which exists, and has the above universal property, since $P/Z < G/Z$ is Zariski closed).
It now suffices to check that the composite map $G\xmaps{\pi} G/Z \xmaps{q} (G/Z)/(P/Z)$ satisfies the universal property of the quotient map $G\to G/P$.  Given a map $f\co G\to X$  whose fibers are all unions of cosets of $P$, the fibers are also unions of cosets of $Z$, so $f$ factors as $f = \overline{f} \circ \pi$.  But now for any $x\in X$, we have
$f^{-1} (x) = \pi^{-1} (\overline{f}^{-1} (x)) = \bigcup_{g\in I} gP$ for some subset $I \subset G$, and hence
$$\overline{f}^{-1} (x) = \pi \left( \pi^{-1} (\overline{f}^{-1} (x)) \right) = \pi \left(  \bigcup_{g\in I} gP \right) = \bigcup_{g\in I} (gZ) P/Z,$$
so the fibers of $\overline{f}$ are unions of cosets of $P/Z$.  Hence $\overline{f}$ factors through the quotient map $q\co G/Z \to (G/Z)/(P/Z)$, yielding the desired factorization of $f$.  Uniqueness of the factorization is immediate.

Combining the previous two paragraphs, we see that if  $Z<P<G$, then $P$ is parabolic if and only if $\pi(P) = P/Z$ is parabolic in $G/Z$, and this completes the proof that $\pi_* (\XC{r} (G)^{red}) = \XC{r} (G/Z)^{red}$.
\end{proof}

\begin{thm}\label{redsingthm}
Let $G$ be a connected, reductive $\C$--group, and let $\wt{DG}$ be the universal cover of the derived subgroup $DG$.  If $\XC{r}(\wt{DG})^{red}\subset \XC{r}(\wt{DG})^{sing}$, then  $\XC{r}(G)^{red}\subset \XC{r}(G)^{sing}.$
\end{thm}

\begin{proof}
Let $[\rho]\in\XC{r}(G)^{red}$.  By Corollary \ref{P-cor}, there exist $[\rho']\in \X_r (DG)^{red}$ and $\chi \in \X_r (T) = T^r$ such that 
$$[\rho] = [([\rho'], \chi)] \in \X_r (DG)^{red} \times_{F^r} T^r \isom \X_r (G)^{red}.$$
Since the kernel of the universal covering homomorphism $\pi \co \wt{DG} \to DG$ is a finite central subgroup of $\wt{DG}$, by Lemma \ref{red-lem} there exists $[\wt{\rho'}]\in \XC{r}(\wt{DG})^{red}$ such that $\pi_* [\wt{\rho'}] = [\rho']$.  By assumption $[\rho']$ is singular, and by Lemma \ref{finitequotientsing} $\pi_* [\wt{\rho'}] = [\rho']$ is also singular.  By Proposition \ref{etale}, we conclude $[\rho]$ is singular.  
\end{proof}

By the work of \cite{FlLa2}, we know that all reducibles in $\XC{r} (\SLm{n})$ are singular for $(r-1)(n-1)\geqs 2$.  Thus, the above theorem immediately implies:

\begin{cor}\label{pslconj}
$\XC{r}(\mathsf{PSL}_n)^{red}\subset \XC{r}(\mathsf{PSL}_n)^{sing}$ for $(r-1)(n-1)\geqs 2$.
\end{cor}

Note that $\mathsf{PSL}_n$ is simple of rank $n-1$.  Hence when $n=2$ we have an example, not covered by Theorem \ref{conj-thm}, in which the above conjecture holds.

\section{Poincar\'e polynomials for $n=2$} \label{PP}

In this section, we show how Tom Baird's computation (see \cite{Ba0})
of the Poincar\'e polynomials of $\XC{r}(\SUm{2})=\SUm{2}^{r}/\SUm{2}$
implies immediately that $\XC{r}(\SUm{2})$ is not a topological manifold
(locally homeomorphic to $\mathbb{R}^{3r-3}$) for $r\geqs 4$. In fact,
we wish to establish that it is also not a topological manifold with
boundary.

To simplify the presentation, consider the following polynomials in
the variable $t$:\begin{eqnarray*}
f_{r}(t) & = & \frac{1}{2}\left[(1+t)^{r}(1+t^{2})-(1-t)^{r}(1-t^{2})\right]\\
h_{r}(t) & = & (1+t^{3})^{r}.\end{eqnarray*}

\begin{prop}[T. Baird, 2008] The Poincar\'e polynomials of $\XC{r}(\SUm{2})$
are:\[
P_{t}(\XC{r}(\SUm{2}))=1+t+\frac{t\ Q(t)}{1-t^{4}},\qquad\mbox{where }Q(t)=t^{2} f_{r}(t)-h_{r}(t).\]

\end{prop}
As particular cases, one can easily compute that, for $r=1,2,3$ and
$4$, we have $P_{t}(\XC{r}(\SUm{2}))=1,1,1+t^{6}$ and $1+4t^{6}+t^{9}$,
respectively.

Let us first check that $P_{t}$ is indeed a polynomial in $t$ with
non-negative integer coefficients. This follows from an alternative
way to write $P_{t}$ which will be useful later. Denote by $\binom{r}{k}$
the binomial coefficient, with the convention that $\binom{r}{k}=0$
for $r<k$.

\begin{lem}\label{Baird}
We have\[
P_{t}(\XC{r}(\SUm{2}))=1+a(t)+b(t),\]
 where $a,b$ are given by the finite series \begin{equation}
a(t)=\sum_{k\geqs 1} \binom{r}{2k+1}t^{2k+4}\frac{1-t^{4k}}{1-t^{4}}\label{eq:a}\end{equation}
 and \begin{equation}
b(t)=\sum_{k\geqs 1} \binom{r}{2k+2}t^{2k+7}\frac{1-t^{4k}}{1-t^{4}},\label{eq:b}\end{equation}
 where \[
\frac{1-t^{4k}}{1-t^{4}}=1+t^{4}+t^{8}+\cdots+t^{4k-4}.\]

\end{lem}
\begin{proof}
We can write $f_{r}(t)=rt+\binom{r}{3}t^{3}+\binom{r}{5}t^{5}+\cdots+t^{2}+\binom{r}{2}t^{4}+\binom{r}{4}t^{6}+\cdots$
and $h_{r}(t)=1+\binom{r}{3}t^{3}+\binom{r}{6}t^{6}+\cdots+t^{3r}$
so that\begin{eqnarray*}
Q(t)=t^{2}f_{r}(t)-h_{r}(t) & = & -1+t^{4}+\binom{r}{3}(t^{5}-t^{9})+\binom{r}{5}(t^{7}-t^{15})+\cdots\\
 &  & \cdots+\binom{r}{4}(t^{8}-t^{12})+\binom{r}{6}(t^{10}-t^{18})+\cdots.\end{eqnarray*}
 Therefore, we get\begin{eqnarray*}
t\ Q(t) & = & -t(1-t^{4})+\binom{r}{3}t^{6}(1-t^{4})+\binom{r}{5}t^{8}(1-t^{8})+\cdots\\
 &  & \cdots+\binom{r}{4}t^{9}(1-t^{4})+\binom{r}{6}t^{11}(1-t^{8})+\cdots\\
 & = & (1-t^{4})(-t+a(t)+b(t))\end{eqnarray*}
 which proves the desired formula. 
\end{proof}
Since $\SUm{2}$ is a compact connected Lie group, the orbit space
$\XC{r}(\SUm{2})=\SUm{2}^{r}/\SUm{2}$ is also a compact connected
topological space, with the natural quotient topology. Observe
that the degree of $P_{t}(\XC{r}(\SUm{2}))=1+t+\frac{t\ Q}{1-t^{4}}$
is given by (for $r\geqs2$) \[
N=\deg Q-3=\max\{\deg f_{r}+2,\deg h_{r}\}-3=3r-3,\]
because the degree of $f_{r}$ is $r+2$ and the degree of $h_{r}$
is $3r$.

\begin{lem}
\label{lem:TopCoef}Let $r\geqs3$ and $N=3r-3$. $(a)$ The Poincar\'e
polynomial $P_{t}$ of $\XC{r}(\SUm{2})$ has degree $N$ and its
top coefficient is $1$. $(b)$ If $\XC{r}(\SUm{2})$ is a manifold
with boundary, then its dimension is $N=3r-3$. 
\end{lem}
\begin{proof}
(a) We have seen that $\deg P_{t}=N=3r-3$. To determine its top coefficient
for $r\geqs3$, note that the top coefficient of $P_{t}$ is either
the top coefficient of $a$, when $r=2k+1$ is odd, or the top coefficient
of $b$, when $r=2k+2$ is even. According to equations (\ref{eq:a})
and (\ref{eq:b}) we have that both the top coefficients of $a$ and
$b$ are 1 (in the odd case, $r=2k+1$, so that $\binom{r}{2k+1}t^{2k+4}t^{4k-4}=1\cdot t^{6k}$,
and $6k=3(r-1)$ and similarly in the even case).

(b) The dimension of $\XC{r}(\SUm{n})$ as a semi-algebraic set is
$(n^{2}-1)(r-1)$ for $r\geqs2$, and if it additionally is a topological
manifold, the dimensions coincide. All semi-algebraic sets have dense
subsets which are manifolds, and it is not hard to see that the irreducible
representations are all smooth and form a dense subset. Clearly, the
projection $\SUm{n}^{r}\to\XC{r}(\SUm{n})$ is locally submersive
at irreducible representations (since their stabilizers are zero dimensional)
and thus the dimension in this case is easily seen to be the dimension
of the tangent space to the representation, $(n^{2}-1)r$ (since $\SUm{n}^{r}$
is smooth), minus the dimension of the orbit, which is $n^{2}-1$,
since the stabilizer is finite. When $n=2$, we get the claim. 
\end{proof}

We will use the following standard facts (see \cite{Hatcher}). By a closed manifold we mean a connected compact
topological manifold (without boundary).

\begin{prop}
\label{pro:Orientable}Let $M$ be an $m$--dimensional closed manifold. If 
$$\dim H_{m}(M,\mathbb{Q})=1,$$ then $M$ is orientable. 
\end{prop}
\begin{thm}[Poincar\'e duality]
\label{thm:PoincareDuality} Let $M$ be a closed
manifold of dimension $m$, and let $b_{k}=\dim H_{k}(M,\mathbb{Q})$
be its $k^{{\rm th}}$ Betti number. If $M$ is orientable, then $b_{k}=b_{m-k}$,
for all $k=0,...,m$. 
\end{thm}
\begin{prop}
\label{pro:Boundary}Let $M$ be a compact connected manifold of dimension
$m$ with non-empty boundary. Then $H_{m}(M,\mathbb{Q})=0$. 
\end{prop}
\begin{cor}
\label{keycorollary} Let $r\geqs3$. If $\XC{r}(\SUm{2})$ is homotopy equivalent to a manifold
with boundary, then $r=3$. 
\end{cor}
\begin{proof}
First we show that the polynomial $P_{t}(\XC{r}(\SUm{2}))=b_{0}+b_{1}t+\cdots+t^{N}$,
where $b_{k}=\dim H_{k}(\XC{r}(\SUm{2}),\mathbb{Q})$, does not satisfy
$b_{k}=b_{N-k}$, when $r\geqs4$. This is clear by looking at the
coefficients. For example, $b_{4}=0$ and, when $r\geqs 5$, equations
(\ref{eq:a}) and (\ref{eq:b}) imply there is always a nonzero coefficient
for the term of degree $N-4=3r-7$, so that $b_{N-4}\neq0$. When $r=4$, we have $P_{t}=1+4t^{6}+t^{9}$
which does not satisfy Poincar\'e duality either. Since Betti numbers are a homotopy invariant, 
we now suppose without loss of generality that $\XC{r}(\SUm{2})$
is a manifold possibly with boundary. Then Lemma \ref{lem:TopCoef}
and Proposition \ref{pro:Boundary} show that $\XC{r}(\SUm{2})$
has no boundary. So, $\XC{r}(\SUm{2})$ is closed of dimension $3r-3$,
and therefore orientable by Proposition \ref{pro:Orientable}. Thus
Poincar\'e duality (Theorem \ref{thm:PoincareDuality}) applies, and
we get a contradiction. So, $\XC{r}(\SUm{2})$ is not a closed manifold
for $r\geqs4$. Since it is connected and compact, it is not everywhere
locally homeomorphic to $\mathbb{R}^{3r-3}$ or to a Euclidean half-space either.
\end{proof}

As a consequence of the above theorem, given the computations in \cite{FlLa2} of the compact and complex $(r,n)=(2,2)$ and $(3,2)$ cases, 
the following corollary holds true.

\begin{cor}
The character variety
$\XC{r}(\Um{2})$, $\XC{r}(\GLm{2})$, $\XC{r}(\SUm{2})$, or $\XC{r}(\SLm{2})$ has the
homotopy type of a manifold if and only if $1\leqs r\leqs 3$.
\end{cor}

Note this corollary is non-trivial since the complex $(3,2)$ case is not a manifold with boundary but it does deformation retract to a manifold.  Moreover, we have shown not only that $\XC{r}(\SUm{2})$ and $\XC{r}(\SLm{2})$ are not generally homotopic to manifolds, but the stronger statement that they are not {\it rational Poincar\'e Duality Spaces}.

\appendix

\section{The proof of Proposition~\ref{trans-prop}}

\subsection{Background}
We will be dealing with simplicial complexes, subdivision, and simplicial maps, so we begin by fixing some terminology and conventions.  For full details, see \cite[Sections 2.1, 2.C]{Hatcher}.  Given a space $Z$, a \e{triangulation} $T$ of $Z$ is a covering of $Z$ by subsets $\sigma \subseteq Z$,  called simplices, equipped with homeomorphisms $\chi_\sigma \co \Delta^k \srm{\isom} \sigma$,
where 
$\Delta^k  \subset \bbR^{k+1}$ is the convex hull of the standard basis vectors $e_1, \ldots, e_{k+1}$.  
The points $\chi_\sigma (e_i)$ are the \e{vertices} of $\sigma$.
This data is subject to the usual axioms.  In particular, a set $U\subset Z$ is open if and only if $U\cap \sigma$ is open in $\sigma$ for all $\sigma \in T$.
Suppressing the characteristic functions from the notation, we write $(Z; T)$ to denote $Z$ equipped with the triangulation $T$, and we refer to this data as a simplicial complex.
We write $\sigma = [t_0, \ldots, t_k]$ to indicate that $\sigma$ is the (unique) simplex in $(Z; T)$ with vertex set $\{t_0, \ldots, t_k\}$.    

A \e{subcomplex} $(Y; T|_Y)$ of $(Z; T)$ is simply a union of simplices in $T$, with the induced triangulation.  The $k$--skeleton of $(Z; T)$ is the subcomplex consisting of all simplices of dimension at most $k$.  Each subcomplex of $(Z; T)$ is closed as a subset of $Z$.  A subcomplex $Y$ of $(Z; T)$ is \e{full} if for every simplex $\sigma = [z_0, \ldots, z_k]\in T$ with $z_i \in Y$ for $i = 0, \ldots, k$, we have $\sigma\subset Y$.

For a  simplex $\sigma$ in a simplicial complex $(Z; T)$, the \e{interior} $\inter{\sigma} = \Inter {\sigma}$ is the subset of $\sigma$ formed by removing all proper faces; equivalently, $\inter{\sigma}$ is the interior of $\sigma$ as a topological manifold with boundary. 
In particular, if $\sigma$ is a vertex, then $\inter{\sigma} = \sigma$.  
Every point in $Z$ lies in the interior of exactly one simplex from $T$.
The \e{boundary} $\partial \sigma = \sigma \sm\inter{\sigma}$ of $\sigma$ is the union of all proper faces of $\sigma$; equivalently, $\partial \sigma$ is the boundary of $\sigma$ as a manifold.

A \e{simplicial map} $(Z; T) \to (Z'; T')$ between simplicial complexes is a map $f\co Z\to Z'$ such that for each $\sigma\in T$, we have $f(\sigma)\in T'$, and
$f|_\sigma$ is linear with respect to the homeomorphisms $\chi_\sigma$ and $\chi_{f(\sigma)}$. 
We note that this implies that 
\begin{equation}\label{simpl-map} f(\inter{\sigma}) = \Inter{f(\sigma)}.\end{equation}

The characteristic functions $\chi_\sigma$ give  us a well-defined notion of the \e{barycenter} $\beta (\sigma)$ for each simplex $\sigma\in T$ and hence a well-defined barycentric subdivision $\beta(T)$ of $T$.
We need a small modification to deal with the fact that a simplicial map $f\co (Z; T)\to (Z'; T')$  does not carry barycenters to barycenters (for instance, consider a simplicial map from a 2--simplex onto a 1--simplex), and hence may no longer be simplicial after barycentrically subdividing the domain and range.
Given any point $b$ in the interior of a $k$--simplex $\sigma$, one may subdivide by barycentrically subdividing all faces of $\sigma$ and then, for each simplex $\tau$ in the subdivision of a face of $\sigma$, adding in the cone on $\tau$ with conepoint $b$ (identifying $\sigma$ with a standard simplex via $\chi_\sigma$, this cone is just the union of all line segments from points in $\tau$ to $b$).  We call this the $b$--centric subdivision of $\sigma$.

Now, if $(Z; T)$ is a 2--dimensional simplicial complex  and 
$$f\co (Z; T) \to (Z'; T')$$ is simplicial, we may choose, for each 2--simplex $\sigma \in T$, a point $b = b(\sigma)$ such that $b\in \Inter{\sigma}$ and $f(b) = \beta(f(\sigma))$.  If $f(\sigma)$ is either $0$-- or $2$--dimensional, then we will always take $b = \beta(\sigma)$ (but if $f(\sigma)$ is an edge, we must choose $b$ differently).
We can now define a subdivision $(Z; S)$ of $(Z; T)$ by taking the barycentric subdivision of the 1--skeleton together with the $b$--centric subdivision of each 2--simplex.  
 We call this an $f$--centric subdivision of $(Z; T)$.  Note that $f$ is simplicial as a map $(Z; S)\to (Z'; T')$.

We need a few standard, elementary lemmas that will be used (sometimes implicitly) throughout this Appendix.  

\begin{lem} \label{int} If $Y\subseteq X$ is a subcomplex of a simplicial complex $(X; T)$ and $\tau\in T$, then $\tau \subset Y$ if and only if   $\inter{\tau}\cap Y \neq \emptyset$.  
\end{lem}

\begin{lem} \label{inv} If $f\co (Z; S) \to (X; T)$ is a simplicial map and $Y$ is a full subcomplex of $(X; T)$, then 
$$f^{-1} (Y) = \bigcup \{ [z_0, \ldots, z_k] \in S \,:\, f(z_i) \in Y \textrm{ for } i = 0, \ldots, k\}.$$
\end{lem}

This follows from Lemma~\ref{int} along with (\ref{simpl-map}).

\begin{lem} $\label{full}$ If $Y\subseteq X$ is a subcomplex of a simplicial complex $(X, T)$, then $Y$ is a \e{full} subcomplex of the barycentric subdivision $(X, \beta (T))$$:$ that is, if $\sigma \in \beta (T)$ is a simplex all of whose vertices lie in $Y$, then $\sigma \subset Y$.
\end{lem}

This also follows from Lemma~\ref{int}, since for each simplex $\sigma \in \beta (T)$ there exists a (unique) simplex
$\tau\in T$ with $\inter{\sigma} \subset \inter{\tau}$, and $\beta(\tau)$ is a vertex of $\sigma$.

\begin{lem}$\label{nbhd}$  Let $K$ be a simplicial complex and let $\sigma \subseteq K$ be a simplex.
Then the set
$$\st (\sigma) = \bigcup \{ \inter{\Delta} \,|\, \sigma \subseteq \Delta\}$$
is open in $K$.
\end{lem}

The set $\st (\sigma)$ is known as the \emph{open star} of $\sigma$.  The proof that $\st(\sigma)$ is open is elementary, and can be found in \cite[Section 62]{Munkres-EAT}.

\begin{lem} $\label{null}$ Let $Z$ be a path connected topological space. Given a nullhomotopic map $f\co S^2\to Z$, the restrictions $f^+$ and $f^-$ of $f$ to the upper and lower hemispheres of $S^2$ are homotopic relative to the equator $S^1 \subseteq S^2$.

Now assume further that $\pi_2 (Z) = 0$.  If $f,g \co  D^2 \to Z$ are continuous maps that agree on $\partial D^2$, then $f\heq g$ $($rel $\partial D^2)$.
\end{lem}

This follows from the fact that nullhomotopic maps from spheres extend to disks.

\subsection{The proof of Proposition~\ref{trans-prop}}

We are given simplicial complexes $(X; T)$ and $(S; N)$, subcomplexes $Y\subset X$ and $S'\subset S$, and a map 
$$f\co (S, S') \to (X, X\sm Y)$$
 that is simplicial on $S'$.  We will assume that $S$ is 2--dimensional; the 1--dimensional case is similar but simpler.
By the Simplicial Approximation Theorem, there exists a  triangulation $N'$ of $S$, with $S'$ still a subcomplex, and a simplicial map $f' \co (S; N') \to (X; T)$ such that $f\heq f'$ (rel $S'$) (see the comments after the proof of the Simplicial Approximation Theorem in \cite[Section 2C]{Hatcher}).  
 
The proof  will proceed through several steps.  Let $N_1$ be an $f$--centric subdivision of $N'$.  In Step 1, we use density of $X\sm Y$ to produce a simplicial map $f_1 \co (S; N_1) \to (X; \beta(T))$ such that $f\heq f_1$ (rel $S'$) and for each 2--simplex $\sigma \in N_1$, $f_1(\sigma) \notsubset Y$.
In Step 2, we use the local connectivity condition around points in $Y$ to construct a map $g \co (S; N_1) \to (X; \beta (T))$ such that $g\heq f_1$ (relative to the union of $S'$ with the vertices of $(S;  N_1)$), and $g^{-1} (Y)$ lies in the vertex set of $(S; N_1)$.
  In Step 3, we use the local $\pi_1$--injectivity condition around points in $g(S) \cap Y$ to homotope $g$ off of $Y$.   

\vspace{.2in}

\noindent \underline{{\bf Step 1:}}

Let $N_1$ be an $f$--centric subdivision of $N'$.  For each $2$--simplex $\sigma  \in N'$, let $b(\sigma)$ denote the new vertex in $N_1$ lying in the interior of $\sigma$.  
We will show that $f$ is homotopic (rel $S'$) to a simplicial map $f_1 \co (S; N_1) \to (X; \beta (T))$ such that for each $2$--simplex $\sigma\in N_1$, we have $f_1 (\sigma)\notsubset Y$.

We will define the desired homotopy separately on each 2--simplex $\sigma \in N$.  Since $X\sm Y$ is dense in $X$, for each simplex $\tau \in T|_Y$ we may choose a simplex $\wt{\tau} \in T$ 
 such that $\tau \subset \wt{\tau}$ and  $\wt{\tau} \nsubset Y$.  
 Now, say $\sigma$ is a 2--simplex in $N$ with $f(\sigma) = \tau \in T|_Y$.  Then there exists a (unique) simplicial map $f^\sigma_1 \co (\sigma; N_1|_\sigma) \to (\wt{\tau}, \beta(T)|_{\wt{\tau}})$ sending $b(\sigma)$ to $\beta (\wt{\tau})$ and sending the other vertices in $(\sigma; N_1|_\sigma)$ to their images under $f$,
and this map agrees with $f$ on $\partial \sigma$.  Since $f|_\sigma$ and $f^\sigma_1$ both map into the contractible space $\wt{\tau}$, Lemma~\ref{null} tells us that $f|_\sigma \heq f^\sigma_1$ (rel $\partial \sigma$).  Moreover, $\Inter{\wt{\tau}} \cap Y = \emptyset$, so we have $(f^\sigma_1)^{-1} (Y) \subset \partial \sigma$.

The maps $f^\sigma_1$ paste together to give a simplicial map 
$$f_1 \co (S; N_1) \to (X, \beta (T))$$
 (which we set equal to $f$ on 2--simplices in $N$ that do not map into $Y$), and we have $f_1 \heq f$ (rel $S'$).  Moreover, for each 2--simplex  $\sigma \in N_1$, we have $f_1 (\sigma) \notsubset Y$.

\vspace{.2in}

\noindent \underline{{\bf Step 2:}}

Next, we homotope $f_1$ further (still rel $S'$) to a  map $g\co  S \to X$ such that $g^{-1}(Y)$ is a subset of the vertices of $N_1$.  We note that it will not be necessary to make $g$ a simplicial map.

Consider an edge $e\in N_1$.  We call $e$ \e{bad} if $f_1 (e) \subset Y$.  Since $\partial S \subset S'$, if $e$ is bad, then $e\notsubset \partial S$ and hence there are exactly two 2--simplices $\sigma(e), \tau(e) \in N_1$ containing $e$.  Let $L(e)$ denote the subcomplex $\sigma(e) \cup \tau(e)$ of 
$(S, N_1)$.
Let $\partial L(e)$ be the union of the edges in $L(e)$ other than $e$ (so $\partial L(e)$ is the boundary of $L(e)$ as a topological disk).

We claim that  if $e$ is bad, then $L(e) \cap f_1^{-1}(Y) = e$, and from this it follows that if
$e_1$ and $e_2$ are distinct bad edges, then $L(e_1) \cap L(e_2) \subset \partial L(e_1)\cap \partial L(e_2)$.
To prove the claim, say $e$ is bad.  Let $s$ and $t$ be the two vertices in $L(e)\sm e$.  By Lemma~\ref{full}, $Y$ is a full subcomplex of $(X, \beta (T))$, and by choice of $f_1$ we have $f_1 (\sigma(e)), f_1(\tau(e)) \notsubset Y$.  Hence we must have $f(s), f(t)\notin Y$, and now Lemma~\ref{inv} proves the claim.

We will now define the desired homotopy of $f_1$ separately on the various sets $L(e)$ with $e$ bad.  
For each such $e$, the following lemma will construct a map $g = g(e) \co L(e)\to X$ such that $g(e) \heq f_1|_{L(e)}$ (rel $\partial L(e)$) and $g(e)^{-1} (Y) = \partial e$.
These homotopies then paste together to define a homotopy $f_1 \heq g$ on all of $S$ (constant outside of these sets), and $g^{-1} (Y)$ will automatically lie inside the vertex set of $(S; N_1)$.  

\begin{lem} $\label{bad-edges}$  Let $(X; T)$ be a simplicial complex and let $(L; R)$ be the simplicial complex consisting of two $2$--simplices $\sigma$ and $\tau$ sharing a common edge $e$.  Let $\partial L$ be the union of the edges in $L$ other than $e$.  Let $Y \subseteq X$ be a full subcomplex such that for each $y\in Y$ and each open neighborhood $U$ of $y$ in $X$, there exists an open neighborhood $V \subset U$, with $y\in V$, such that $V\sm Y$ is path connected.
If $f\co (L; R) \to (X; T)$ is a simplicial map such that $f^{-1} (Y) = e$, then $f$ is homotopic $($rel  $\partial L$$)$ to a map $g\co L\to X$ such that $g^{-1} (Y)  = \partial e$.  
\end{lem}
\begin{proof} 
Let $U\subset X$ be the open set  $\st(f(e))$.  
Since $\beta(f(e)) \in Y$, there exists an open set $V\subset X$ such that $\beta(f(e))\in V \subset U$ and $V\sm Y$ is path connected.  

We will say that a point $r \in V\sm Y$ is \e{reachable} (from $f(\sigma)$) if there exists a sequence of simplices $\Delta_1, \ldots, \Delta_m \in T$ satisfying
\begin{enumerate}[{\bfseries (i)}]
\item $f(\sigma) \subseteq \Delta_1$,
\item $f(e) \subseteq \Delta_i$ for $1\leqs  i\leqs  m$, 
\item $\Delta_i \cap \Delta_{i+1} \nsubset Y$ for $1\leqs  i\leqs  m-1$,
\item $r \in \inter{\Delta}_m$.
\end{enumerate}
Note that condition {\bf (iii)} is vacuous when $m=1$.
 
We now argue that all points in $V\sm Y$ are reachable.
Note that  $V\cap f(\sigma) \neq \emptyset$ (since $\beta(f(e))\in V\cap f(\sigma)$),  so openness of $V$ implies that
$\Inter {f(\sigma)}\cap V \neq \emptyset$.  Moreover, each point in $\Inter {f(\sigma)}\cap V$ is reachable, by setting $m=1$ and $\Delta_1 = f(\sigma)$.  So the set of reachable points is non-empty.

Next, we check that the set of reachable points is open.  If $r\in V\sm Y$ is reachable, then there exists a sequence of simplices $\Delta_1, \ldots, \Delta_m$ as above.  Since $r\in \inter{\Delta}_m \subset \st(\Delta_m)$,  to prove openness it will suffice to show that all points in $(V\sm Y) \cap \st(\Delta_m)$ are reachable.  If $x\in (V\sm Y) \cap \st(\Delta_m)$, then $x\in \inter{\Delta}_{m+1}$ for some simplex $\Delta_{m+1} \in T$ with $\Delta_m \subset \Delta_{m+1}$.  We claim that the sequence $\Delta_1, \ldots, \Delta_m, \Delta_{m+1}$ 
satisfies the above conditions, thereby proving $x$ is reachable.
We must check that $\Delta_m \cap \Delta_{m+1} \nsubset Y$; all the other conditions are immediate from our choices of $\Delta_1, \ldots, \Delta_{m+1}$.  Note that $\Delta_m \cap \Delta_{m+1} = \Delta_m$.  
If $m = 1$, then by {\bf (i)} we have $f(\sigma) \subset \Delta_m$, so $\Delta_m \nsubset Y$; on the other hand, if $m > 1$, then $\Delta_{m-1} \cap \Delta_m \nsubset Y$, so $\Delta_m \nsubset Y$.

Finally, we show that  the set of non-reachable points in $V\sm Y$ is also open.  Since $V\sm Y$ is (path) connected, this will imply that all points in $V\sm Y$ are reachable.  Say $v\in V\sm Y$ is not reachable.  Since $v\in V\subset U = \st(f(e))$, we have $v\in \inter{\Delta}$ for some simplex $\Delta \in T$ with $f(e) \subset \Delta$.  We claim that no point in $(V\sm Y) \cap \st (\Delta)$ is reachable; this will suffice since this open set contains $v$.  If $w\in (V\sm Y) \cap \st (\Delta)$ were reachable, then we would have a sequence $\Delta_1, \ldots, \Delta_m$ as above, with $w \in \inter{\Delta}_m$.  Since $w\in \st (\Delta)$, we have $\Delta \subset \Delta_m$.  We claim that the sequence $\Delta_1, \ldots, \Delta_m, \Delta$ shows that $v$ is reachable, a contradiction.  
All that remains to be checked is that $\Delta_m \cap \Delta \nsubset Y$.  
But $\Delta_m \cap \Delta = \Delta$, and $\Delta\nsubset Y$ because $v\in \Delta$ and $v\in V\sm Y$.

To complete the proof of the lemma, we will need to introduce some auxiliary simplicial complexes.  Let $D^2_m$ be 
geometric realization (see \cite[Section 3]{Munkres-EAT}) of  the abstract 2--dimensional simplicial complex with $m+2$ vertices $w_0, v_1, \ldots, v_{m+1}$, and 2--simplices $[w_0, v_i, v_{i+1}]$ for $i = 1, \ldots, m$.  The edges in $D^2_m$ are simply the faces of these 2--simplices.
Let $D^3_m$ be the geometric realization of  the abstract 3--dimensional simplicial complex with $m+3$ vertices $w_0, w_1, v_1, \ldots, v_{m+1}$, and 3--simplices $[w_0, w_1, v_i, v_{i+1}]$ for $i = 1, \ldots, m$.  The 1-- and 2--simplices in $D^3_m$ are simply the faces of these 3--simplices.
Note that $D^2_m$ is homeomorphic to a topological 2--disk, and $D^3_m$ is homeomorphic to a topological 3--disk.  In particular, both spaces are contractible.

Now, $\beta (f(e)) \in V\cap f(\tau)$,  so 
$\Inter {f(\tau)}\cap V \neq \emptyset$.  Given $x\in \Inter {f(\tau)}\cap V$, we know that $x$ is reachable, so there exists a  sequence of simplices $\Delta_1, \ldots, \Delta_m \in T$ satisfying the above properties {\bf (i)}-{\bf(iv)}.   
Note that by {\bf(iv)}, we have $\Inter{f(\tau)} \cap \Inter{\Delta_m}\neq \emptyset$, so $\Delta_m = f(\tau)$.

Let $e = [\epsilon_0, \epsilon_1]$, $\sigma = [\epsilon_0, \epsilon_1, s]$, and $\tau = [\epsilon_0, \epsilon_1, t]$.
Let $v_1' = f(s)$ and let $v_{m+1}' = f(t)$, and note that $v_1' \in \Delta_1$ by {\bf (i)}, and $v_{m+1}' \in \Delta_m = f(\tau)$.
Since $Y$ is a full subcomplex of $X$,  {\bf (iii)} implies that there exist vertices $v'_{i} \in \Delta_{i-1} \cap \Delta_{i}$ ($i=2, \ldots, m$) such that $v'_i\notin Y$.  Since $f^{-1} (Y) = e$, we also have $v_1', v_{m+1}'\notin Y$.
 Let $j = \dim f(e) + 2$.  
 For $i = 1, \ldots, m$, we have $f(e) \cup \{v'_i, v'_{i+1}\} \subset \Delta_i$ (here we are using {\bf (ii)}), so there exists a (unique) simplicial map $p\co D^j_m \to (X, T)$ defined by $p (w_i) = f(\epsilon_i)$ and $p(v_i) = v_i'$.  
Also, we have a simplicial map $q\co L \to D^j_m$ defined by $q (\epsilon_i) = w_i$, $q(s) = v_1$, and $q(t) = v_{m+1}$.  Now   $f|_{L}\co L\to X$ factors as $p\circ q$.

Since $D^j_m$ is a topological disk, the loop $q(\partial L)$ (which lies on the boundary of this disk) is nullhomotopic via a homotopy that lies in the interior of $D^j_m$ except at time 0 (for instance, choosing a homeomorphism between  $D^j_m$ and the unit disk $D^j \subset \bbR^{j}$, we can radially contract $q(\partial L)$ to the origin).  This homotopy gives a map $q'\co L \isom D^2 \to D^j_m$, which (by Lemma~\ref{null}) is homotopic to $q$ (rel $\partial L$).  We note that $(q')^{-1} ([w_0, w_{\dim f(e)}]) = \partial e$.

Now we have $f = p\circ q \heq p \circ q'$ (rel $\partial L$), and to complete the proof it will suffice to check that $p\circ q'$ satisfies $(p\circ q')^{-1} (Y) = \partial e$.  We have $(p\circ q')^{-1} (Y) = (q')^{-1} \left(p^{-1} (Y)\right)$, so it is enough to show that $p^{-1} (Y)  = [w_0, w_{\dim f(e)}]$.  But $p(v_i) = v_i'\notin Y$ by choice of $v_i'$, and so the only vertices of $D^j_m$ that map to $Y$ under $p$ are those of the form $w_i$.  
Applying Lemma~\ref{inv} completes the proof.
\end{proof}

\noindent \underline{{\bf Step 3:}}

At this point, we have produced a map $g$, homotopic to $f$ (rel $S'$), such that $g^{-1} (Y)$ consists only of vertices in $(S; N_1)$.  Let $W$ be the set of vertices in $N_1$ that map to $Y$ under $g$, and note that $W\cap S' = \emptyset$.  
For each $w\in W$, our hypotheses on $X$ state that there exists a path connected open neighborhood $V_w$ of $g(w)$ such that $\pi_2 (V_w) = 0$ and for all $v\in V_w\sm Y$, the natural map $\pi_1 (V_w\setminus Y, v) \to \pi_1 (V_w,v)$ is injective.  For each $w\in W$, choose a Euclidean neighborhood $U_w \subset S$ containing $w$.  Then there exists a disk $D_w \subset U_w \cap  g^{-1} (V_w)$, centered at $w$, such that   $D_w$ is disjoint from $S'$ and does not contain any vertex of $N_1$ other than $w$.  Note that by shrinking these disks if necessary, we may assume that they are disjoint.  Now $g(D_w \sm \{w\}) \subset V_w \sm Y$, so $\partial D_w$ maps under $g$ to a loop in $V_w\sm Y$ that is nullhomotopic in $V_w$.  For each $z\in \partial D_w$, we have $\pi_1 (V_w \sm Y, g(z)) \injects \pi_1 (V_w , g(z))$, so each of these loops is in fact nullhomotopic in $V_w \sm Y$, and this nullhomotopy defines a map $h_w \co D_w \to V_w \sm Y$ such that $h_w|_{\partial D_w} = g|_{\partial D_w}$.  Since $\pi_2 (V_w) = 0$, Lemma~\ref{null} implies that $h_w 
\heq g|_{D_w}$ (rel $\partial D_w$).  We can now define a new map $h\co S \to X$ by replacing $g$ with $h_w$ on each $D_w$, and we have $h\heq g$ (rel $S'$), since $S'\cap \left(\bigcup_{w\in W} D_w\right) = \emptyset$.
Since $h(S) \subset X\sm Y$, this completes the proof of Proposition~\ref{trans-prop}.
$\hfill \Box$

\def\cdprime{$''$} \def\cprime{$'$} \def\cprime{$'$} \def\cprime{$'$}
\providecommand{\bysame}{\leavevmode\hbox to3em{\hrulefill}\thinspace}
\providecommand{\MR}{\relax\ifhmode\unskip\space\fi MR }
\providecommand{\MRhref}[2]{%
  \href{http://www.ams.org/mathscinet-getitem?mr=#1}{#2}
}
\providecommand{\href}[2]{#2}

\end{document}